\documentclass[review,onefignum,onetabnum]{siamart171218}

\usepackage{lipsum}
\usepackage{amsfonts}
\usepackage{graphicx}
\usepackage{epstopdf}
\usepackage{algpseudocode}
\usepackage{enumerate}
\usepackage{graphicx} 
\usepackage{amsmath}
\usepackage{amssymb}
\usepackage{tikz}
\usepackage{caption}
\usepackage{subcaption}
\usepackage{pgfplots} 
\usepackage{pgfgantt}
\usepackage{pdflscape}
\ifpdf
  \DeclareGraphicsExtensions{.eps,.pdf,.png,.jpg}
\else
  \DeclareGraphicsExtensions{.eps}
\fi

\newsiamremark{remark}{Remark}
\newsiamremark{hypothesis}{Hypothesis}
\crefname{hypothesis}{Hypothesis}{Hypotheses}
\newsiamthm{claim}{Claim}

\newtheorem{assumption}{Assumption}%

\headers{$\mathcal{H}_2$ optimal rational approximation on general domains}{A. Borghi and T. Breiten}

\title{$\mathcal{H}_2$ optimal rational approximation on general domains}

\author{Alessandro Borghi\thanks{Institut für Mathematik, Technische Universit\"at Berlin, 10623 Berlin, Germany 
  (\email{borghi@tu-berlin.de}, \url{https://www.tu.berlin/fgmso/alessandro-borghi}).}
\and Tobias Breiten\thanks{Institut für Mathematik, Technische Universit\"at Berlin, 10623 Berlin, Germany 
  (\email{tobias.breiten@tu-berlin.de}, \url{https://www.tu.berlin/fgmso/tobias-breiten})}
}

\usepackage{amsopn}

\makeatletter
\newcommand*{\addFileDependency}[1]{
  \typeout{(#1)}
  \@addtofilelist{#1}
  \IfFileExists{#1}{}{\typeout{No file #1.}}
}
\makeatother

\newcommand\myatop[2]{\genfrac{}{}{0pt}{}{#1}{#2}}

\begin{document}
\nolinenumbers

\maketitle

\begin{abstract}
  Optimal model reduction for large-scale linear dynamical systems is studied. In contrast to most existing works, the systems under consideration are not required to be stable, neither in discrete nor in continuous time. As a consequence, the underlying rational transfer functions are allowed to have poles in general domains in the complex plane. In particular, this covers the case of specific conservative partial differential equations such as the linear Schrödinger and the undamped linear wave equation with spectra on the imaginary axis. 
    By an appropriate modification of the classical continuous time Hardy space $\mathcal{H}_2$, a new $\mathcal{H}_2$ like optimal model reduction problem is introduced and first order optimality conditions are derived. As in the classical $\mathcal{H}_2$ case, these conditions exhibit a rational Hermite interpolation structure for which an iterative model reduction algorithm is proposed. Numerical examples demonstrate the effectiveness of the new method.
\end{abstract}



\section{Introduction} \label{sec:1}

We consider large-scale single input single output (SISO) linear time invariant (LTI) dynamical systems of the form
\begin{equation}\label{eq:fom}
	\begin{cases}
		\dot{\mathbf{x}}(t)=\mathbf{A}\mathbf{x}(t)+\mathbf{b}u(t),\quad \mathbf{x}(0)=0,\\
		y(t) = \mathbf{c}^* \mathbf{x}(t),    \end{cases}
\end{equation}
where $\mathbf{A}\in\mathbb{C}^{n\times n}$ and $\mathbf{c},\mathbf{b}\in\mathbb{C}^{n}$. Here, for fixed time $t$, $\mathbf{x}(t)\in\mathbb{C}^{n}$, $u(t)\in\mathbb{C}$, and $y(t)\in\mathbb{C}$ are the state, input, and output of the system, respectively. Let us emphasize that the results in this article similarly hold true for systems with multiple inputs and outputs and the restriction to the SISO case is made for the ease of presentation. Accompanied by the time domain description \eqref{eq:fom} we have the (frequency domain) transfer function 
\begin{equation*}
	H(s)=\mathbf{c}^*(s\mathbf{I} - \mathbf{A})^{-1} \mathbf{b}.
\end{equation*}
If \eqref{eq:fom} is minimal, then $H$ is a rational function of degree $n$. As an accurate modeling of such systems in both time and frequency domain can be computationally expensive for large values of $n$, we are interested in the construction of a reduced order surrogate model of the form     
\begin{equation} \label{eq:rom}
	\begin{cases}
		\dot{\widehat{\mathbf{x}}}_r(t)=\widehat{\mathbf{A}}_r\widehat{\mathbf{x}}_r(t)+\widehat{\mathbf{b}}_ru(t),\quad \widehat{\mathbf{x}}_r(0)=0,\\
		\widehat{y}_r(t) = \widehat{\mathbf{c}}_r^* \widehat{\mathbf{x}}_r(t),    \end{cases}
\end{equation}
with transfer function 
\begin{equation*}
	\widehat{H}(s)=\widehat{\mathbf{c}}_r^*(s\mathbf{I} - \widehat{\mathbf{A}}_r)^{-1} \widehat{\mathbf{b}}_r,
\end{equation*}
where $\widehat{\mathbf{A}}_r\in\mathbb{C}^{r\times r}$ and $\widehat{\mathbf{c}}_r,\widehat{\mathbf{b}}_r\in\mathbb{C}^{r}$. Here, the goal for constructing a ``good'' reduced order model is twofold: on the one hand, we are interested in \eqref{eq:rom} being efficiently solvable such that we demand $r\ll n$; on the other hand, the reduced model is expected to yield an accurate approximation such that the outputs of both full and reduced system are close to each other, i.e., $\widehat{y}_r(t)\approx y(t)$ for  $t\ge 0$. The latter condition requires a more precise notion of similarity. For example, it is well known, see, e.g., \cite{antoulas2005} that 
\begin{align}\label{eq:linf_l2}
  \sup\limits_{t\ge 0} \,  \lvert y(t)-\widehat{y}_r(t)\lvert
  \le \| H-\widehat{H}\|_{\mathcal{H}_2} \|u\|_{\mathcal{L}_2}
\end{align}
which has led to the study of $\mathcal{H}_2$ optimal model reduction problems, see \cite{gugercin2008,meier1967}. More generally, model reduction techniques for linear systems of the form \eqref{eq:fom} have been addressed from a multitude of different areas such as system theory \cite{mullis1976,moore1981,meier1967}, reduced basis methods \cite{barrault2004,rozza2008}, rational interpolation \cite{grimme97,freund2003,mayo2007}, proper orthogonal decomposition \cite{sirovich1987,kunish2001} and, more recently, data driven techniques \cite{brunton2016,kutz2016,peherstorfer2016}. While a complete overview of the existing literature is out of the scope of this article, let us refer to the monographs \cite{antoulas2005,hesthaven2015,benner2017,benner2005} and the references therein.     

The results discussed throughout this article are related to the error bound in \eqref{eq:linf_l2} and the corresponding $\mathcal{H}_2$ approximation problem of the underlying transfer functions. First order  optimality conditions for $\mathcal{H}_2$ model reduction of linear systems have already been derived in \cite{meier1967,wilson1970}. Later on, in \cite{gugercin2008} the iterative rational Krylov algorithm (IRKA) has been proposed to numerically compute $\mathcal{H}_2$ optimal reduced order models, see also \cite{vandooren2008,vandooren2010}. A similar method, called MIRIAm, has been discussed in a discrete time setting in \cite{gerstner2010}. Several extensions for frequency-weighted \cite{anic2013,breiten2015}, frequency-limited \cite{petersson2014}, structure-preserving \cite{gugercin2009}, parametric \cite{baur2011} or data-driven \cite{beattie2012} $\mathcal{H}_2$ problems have been developed over the last years. One of the essential theoretical assumptions made in these works is that the full order model \eqref{eq:fom} is (asymptotically) stable. For unstable systems, available methods are rather scarce. Some notable exceptions are the approaches discussed in \cite{magrauder2010,sinani2019}. Moreover, let us particularly mention the $h_{2,\alpha}$ model reduction technique proposed in \cite{kubalinska2008} which allows to treat discrete time systems with system poles outside of the unit circle by extending the classical (discrete time) Hardy space to a circle of radius $\alpha$. The strategy we follow here is based on similar ideas and also relates to the recent more general $\mathcal{L}_2$ optimal model reduction framework from \cite{mlinaric2022,mlinaric2023}. We build upon the existing theory of optimal $\mathcal{H}_2$ model reduction and extend it by an appropriate use of specific conformal maps. 
The main contributions are the following:
\begin{enumerate}[(i)]
	\item We consider rational transfer functions with poles in specific domains in $\mathbb C$, thereby covering typical cases such as the open left half plane and the open unit disk. For this purpose, we define the $\mathcal{H}_2(\bar{\mathbb{A}}^{\mathsf{c}})$ space in Definition \ref{def:H2A}.
	\item We study the resulting  $\mathcal{H}_2(\bar{\mathbb{A}}^{\mathsf{c}})$ optimal model reduction problem and derive structured first order necessary interpolation optimality conditions in Theorem \ref{theorem:main}.
	\item Under additional assumptions, we show in Corollary \ref{corollary:practicaltheorem} that more explicit interpolation conditions can be obtained for which we propose a numerical algorithm by a modification of IRKA, see Algorithm \ref{algoirka}.
	\item The proposed algorithm is shown to be applicable to the Schr\"odinger and the undamped wave equation where the system poles are positioned along the imaginary axis.\\
\end{enumerate}

The rest of the paper is organized as follows. In Section \ref{sec1} we briefly recall the concept of interpolatory model order reduction and existing interpolation-based $\mathcal{H}_2$ optimality conditions. In Section \ref{sec:H2Abar} we introduce a Hardy space for functions with poles in general domains. We refer to this space as  $\mathcal{H}_2(\bar{\mathbb{A}}^{\mathsf{c}})$. We discuss its connection to the classical $\mathcal{H}_2$ space and characterize the arising inner products.  
In Section \ref{sec:optconditions} we introduce an $\mathcal{H}_2(\bar{\mathbb{A}}^{\mathsf{c}})$ optimal model reduction problem for which we derive first order necessary interpolation conditions for local optimality. In addition, under specific assumptions, we propose a method based on IRKA for computing a solution to the model reduction problem. In Section \ref{sec:numexp} we demonstrate the effectiveness of our approach through numerical experiments for two spatially discretized partial differential equations with spectra residing along the imaginary axis.

\subsection{Notation}\label{sec:notation}

By $\mathbb{C}$, $\mathbb{C}_+$, $\mathbb{C}_-$ and $\mathbb{D}$ we denote the complex plane, the open right half and open left half complex plane, and the open unit disk, respectively.
For an open subset $\mathbb{X}\subset\mathbb{C}$, we denote $\delta\mathbb{X}$ to be its boundary, $\mathbb{X}^{\mathsf{c}}$ its complement, $\bar{\mathbb{X}}$ its closure, and $\bar{\mathbb{X}}^{\mathsf{c}}=\mathbb{C}\backslash\left\{\mathbb{X}\cup\delta\mathbb{X}\right\}$ its exterior. We denote the complex conjugation of a scalar and the Hermitian of a matrix by $(\cdot)^*$. We denote the Euclidean norm by $\|\cdot\|_2$ and the absolute value over the complex numbers by $\lvert x\lvert=\sqrt{xx^*}$ where $x\in\mathbb{C}$. The symbol $\mathrm{i}$ denotes the imaginary unit. Let $\mathbf{X}\in\mathbb{C}^{n\times n}$ be a matrix, then $\text{Ran}(\mathbf{X})$ is its range. The first and second derivative of the function $f$ at a point $x$, i.e., $\left[\frac{\mathrm{d}}{\mathrm{d}s}f(s)\right]_{s=x}$ and $\left[\frac{\mathrm{d}^2}{\mathrm{d}s^2}f(s)\right]_{s=x}$ , are denoted by $f'(x)$ and $f''(x)$, respectively. For two complex valued functions, $f$ and $g$, we denote their composition by $f\circ g$ and its evaluation at $x$ by $f(g(x))$. The inverse of a function $f$ is denoted by $f^{-1}$. Let $f$ be analytic with a Taylor series around $x_0$ equal to $f(x)=\sum_{j=0}^\infty a_j(x-x_0)^j$. Then $\overline{f}(x)=\sum_{j=0}^\infty a_j^*(x-x_0^*)^j=f(x^*)^*$. In other words, $\overline{f}$ is $f$ but with its coefficients replaced by their complex conjugates.
For a rational function $F$ with poles $\{\lambda_j\}_{j=1}^n\in \mathbb{X}$, we denote the residue of $ F$ in $\lambda$ by $\mathsf{res}[ F(s),\lambda]$. If $\lambda$ is a simple pole, then $\mathsf{res}[ F(s),\lambda]=\lim\limits_{s\rightarrow\lambda}(s-\lambda) F(s)$, if $\lambda$ is a double pole, then $\mathsf{res}[ F(s),\lambda]=\lim\limits_{s\rightarrow\lambda}\left[\frac{\mathrm{d}}{\mathrm{d}x}(x-\lambda)^2 F(x)\right]_{x=s}$.  

\section{Interpolatory model order reduction}\label{sec1}
In this section, we briefly recall the concept of interpolatory model reduction. In particular, we summarize the problem of $\mathcal{H}_2$ optimal model reduction and how it relates to rational (Hermite) interpolation as it lays the foundations for our main results in Section \ref{sec:H2Abar}.
For the construction of the reduced model in \eqref{eq:rom}, we consider a Petrov-Galerkin projection. In other words, given  two matrices $\mathbf{V}_r,\mathbf{W}_r\in\mathbb{C}^{n\times r}$, we consider an approximation of the form $\mathbf{x}(t)\approx\mathbf{V}_r\widehat{\mathbf{x}}_r(t)$ such that the residual for \eqref{eq:fom} satisfies the following orthogonality condition
\[
\mathrm{Ran}(\mathbf{W}_r) \perp \left(\mathbf{V}_r\dot{\widehat{\mathbf{x}}}_r(t)-\mathbf{A}\mathbf{V}_r\widehat{\mathbf{x}}_r(t)-\mathbf{b}u(t)\right).
\]
If $\mathbf{W}_r^*\mathbf{V}_r$ is invertible, this leads to the reduced order system matrices
\begin{equation}\label{eq:systemmatrices}
	\widehat{\mathbf{A}}_r = (\mathbf{W}_r^*\mathbf{V}_r)^{-1}\mathbf{W}_r^*\mathbf{A}\mathbf{V}_r,\; \widehat{\mathbf{b}}_r = (\mathbf{W}_r^*\mathbf{V}_r)^{-1}\mathbf{W}_r^*\mathbf{b},\; \textnormal{and } \widehat{\mathbf{c}}_r^* = \mathbf{c}^*\mathbf{V}_r.
\end{equation}
It is well-known, see, e.g., \cite{grimme97,DeVillemagne1987,gugercin2008} that by choosing $\mathbf{V}_r$ and $\mathbf{W}_r$ as rational Krylov subspaces characterized by the resolvent operator $(\sigma I-\mathbf{A})^{-1}$, the reduced order model \eqref{eq:rom} satisfies the following Hermite type interpolation conditions.
\begin{theorem}(\cite[Lemma 2.1]{gugercin2008})\label{theorem:projection}
	Consider the transfer function $H$ of the full order model in (\ref{eq:fom}) with matrices $\mathbf{A},\mathbf{b},\mathbf{c}$. Consider also the interpolation points $\{\sigma_j\}_{j=1}^r$ such that $(\sigma_j\mathbf{I}-\mathbf{A})$ and $(\sigma_j\mathbf{I}-\widehat{\mathbf{A}}_r)$ are both nonsingular. Let the two projection matrices $\mathbf{V}_r$ and $\mathbf{W}_r$ be chosen such that  
	\begin{equation} \label{eq:projectionmatrices}
		\begin{aligned}
			\textnormal{Ran}(\mathbf{V}_r)&=\textnormal{span}\left\{(\sigma_1\mathbf{I}-\mathbf{A})^{-1}\mathbf{b},\dots, (\sigma_r\mathbf{I}-\mathbf{A})^{-1}\mathbf{b} \right\},\\
			\textnormal{Ran}(\mathbf{W}_r)&=\textnormal{span}\left\{(\sigma_1^*\mathbf{I}-\mathbf{A}^*)^{-1}\mathbf{c},\dots, (\sigma_r^*\mathbf{I}-\mathbf{A}^*)^{-1}\mathbf{c} \right\}.
		\end{aligned}
	\end{equation}
	Then the reduced transfer function $\widehat{H}$ with matrices as in (\ref{eq:systemmatrices}) satisfies
	\begin{equation*}
		\widehat{H}(\sigma_j)  = H (\sigma_j) \; \textnormal{and } \widehat{H}'(\sigma_j)  = H' (\sigma_j) \; \textnormal{for } j=1,\dots,r.
	\end{equation*}
\end{theorem}

Since the choice of interpolation points has a significant influence on the quality of the reduced model, different selection strategies for $\sigma_j$ have been proposed. For our purposes, so-called $\mathcal{H}_2$ optimal interpolation points, see \cite{gugercin2008,gerstner2010}, will be of particular relevance.

\subsection{Optimal \texorpdfstring{$\mathcal{H}_2$}{TEXT} model reduction}\label{sec:H2}

Recall that for functions $G,H$ that are analytic in the open right half plane, the Hardy space $\mathcal{H}_2(\mathbb{C}_+)$ (\cite[Section 5.1.3]{antoulas2005}) is defined as 
\begin{equation*}
	\mathcal{H}_2(\mathbb{C}_+):=\left\{ H\colon \mathbb{C}_+\rightarrow \mathbb{C} \;\text{analytic}\;\Big\lvert \sup_{x>0}\int_{-\infty}^{\infty}\left\lvert H(x+\mathrm{i} \omega)\right\lvert^2\mathrm{d}\omega < \infty \right\}.
\end{equation*}
Moreover, $\mathcal{H}_2(\mathbb{C}_+)$ becomes a Hilbert space when endowed with the inner product 
\begin{equation}\label{eq:H2Cplus}
	\left\langle H,G\right\rangle_{\mathcal{H}_2(\mathbb{C}_+)}:=\frac{1}{2\pi}\int_{-\infty}^{\infty} H(\mathrm{i} \omega)^* G(\mathrm{i} \omega)\,\mathrm{d}\omega,
\end{equation}
and induced norm
\begin{equation*}
	\| H\|_{\mathcal{H}_2(\mathbb{C}_+)}:=\left(\frac{1}{2\pi}\int_{-\infty}^{\infty}\left\lvert H(\mathrm{i} \omega)\right\lvert^2\mathrm{d}\omega\right)^{\frac{1}{2}}.
\end{equation*}
Given a transfer function $H$ of a continuous time asymptotically stable system with poles in the open left half plane, the $\mathcal{H}_2$ optimal model reduction problem is 
\begin{equation}\label{eq:H2optprob}
	\min\limits_{\myatop{\tilde{H} \text{ asymp.~stable}}{\mathrm{dim}(\tilde{H})=r}}
\|H-\tilde{H}\|_{\mathcal{H}_2(\mathbb{C}_+)}.
\end{equation}
In \cite{meier1967,gugercin2008} it was proved that if $\widehat{H}$, with poles $\{\widehat{\lambda}_j\}_{j=1}^r$, is a local minimizer of (\ref{eq:H2optprob}), then 
\begin{equation}\label{eq:H2Coptimalconditions}
	\widehat{H}(-\widehat{\lambda}_j^*)=H(-\widehat{\lambda}_j^*)\text{ and } \widehat{H}'(-\widehat{\lambda}_j^*)=H'(-\widehat{\lambda}_j^*)\text{ for }j=1,\dots,r.
\end{equation}
These interpolation conditions are usually referred to as \textit{Meier-Luenberger conditions}, see \cite{gugercin2008}.
A similar result was proved for discrete time systems. In this case, the transfer functions $H$ and $\tilde{H}$ are analytic in $\bar{\mathbb{D}}^{\mathsf{c}}$ (see \cite{gugercin2008,gerstner2010}). The optimal $\mathcal{H}_2$ framework  minimizes the error norm
\begin{equation*}
	\| H-\tilde{H}\|_{\mathcal{H}_2(\bar{\mathbb{D}}^{\mathsf{c}})}:=\left(\frac{1}{2\pi}\int_{0}^{2\pi}\left\lvert H(e^{\mathrm{i}\vartheta})-\tilde{H}(e^{\mathrm{i}\vartheta})\right\lvert^2\mathrm{d}\vartheta\right)^{\frac{1}{2}},
\end{equation*}
and for $\widehat{H}$ being a local minimizer we have that
\begin{equation}\label{eq:h2Doptimalconditions}
	\widehat{H}\left(1/\widehat{\lambda}_j^*\right)=H\left(1/\widehat{\lambda}_j^*\right)\text{ and } \widehat{H}'\left(1/\widehat{\lambda}_j^*\right)=H'\left(1/\widehat{\lambda}_j^*\right)\text{ for }j=1,\dots,r.
\end{equation}
Throughout this manuscript, we refer to  (\ref{eq:H2Coptimalconditions}) and (\ref{eq:h2Doptimalconditions}) as the $\mathcal{H}_2(\mathbb{C}_+)$ and $\mathcal{H}_2(\bar{\mathbb{D}}^{\mathsf{c}})$ optimality conditions respectively, thereby emphasizing in which set the functions are analytic. This will turn out useful for the upcoming results. 
 
 Note that as long as the systems are assumed to be asymptotically stable, considering Hardy spaces on the specific sets $\mathbb{C}_+$ and $\bar{\mathbb{D}}^{\mathsf{c}}$ is sufficient. For cases in which the poles reside in different sets or when a more granular view of the spectrum is desired, the conditions (\ref{eq:H2Coptimalconditions}) and (\ref{eq:h2Doptimalconditions}) may not be an ideal choice. In this context, \cite{kubalinska2008} discussed generalizations of (\ref{eq:H2Coptimalconditions}) and (\ref{eq:h2Doptimalconditions}) to the shifted right half plane and disks with arbitrary radii, thereby allowing for unstable models. 
Although these findings can be applied to a wider set of functions, they still restrict one to a very specific shape of the considered domains.
In the next section we show that it is possible to generalize these frameworks to rational transfer functions with poles located in general domains, i.e., non-empty connected open sets.

\section{The \texorpdfstring{$\mathcal{H}_2(\bar{\mathbb{A}
}^{\mathsf{c}})$}{TEXT} space}\label{sec:H2Abar}

Consider $\mathbb{A}\subset \mathbb{C}$ to be a non-empty connected open set in the complex plane.
We define $ F$ and $ G$ as the rational functions
\begin{equation}\label{eq:functionbasic}
	 F(s) = \sum_{i=1}^n \frac{\phi_i}{s-\lambda_i}, \quad G(s) = \sum_{j=1}^q \frac{\nu_j}{s-\mu_j}
\end{equation}
where $\phi_i=\mathsf{res}[ F(s),\lambda_i]$ and $\nu_j=\mathsf{res}[ G(s),\mu_j]$. In particular, we assume both $F$ and $G$ to have only simple poles.

Throughout this section we make extensive use of conformal maps and therefore recall the conformal mapping theorem. 
\begin{theorem}(\cite[Theorem 6.1.2]{Wegert2012})\label{th:com} 
Suppose $\mathbb{X},\mathbb{Y}\subset\mathbb{C}$ are open sets and let $\psi\colon\mathbb{X}\rightarrow\mathbb{Y}$ be Fr\'echet differentiable as a function of two real variables. The mapping $\psi$ is conformal in $\mathbb{X}$ if and only if it is analytic in $\mathbb{X}$ and $\psi'(s_0)\neq 0$ for every $s_0\in\mathbb{X}$.
\end{theorem}
Similar statements on the properties of conformal maps can be found in \cite[Theorem I.5.15]{freitag2006}.
From \cite[Section 2.6]{kythe2019} and \cite[Section 1.1]{dmitry2019} we also have the following properties: (\textit{i}) Since $\psi$ is analytic in $\mathbb{X}$ and $\psi'$ does not vanish, $\psi$ is injective in $\mathbb{X}$. (\textit{ii}) If $\psi\colon\mathbb{X}\rightarrow\mathbb{Y}$ is conformal and bijective, then also its inverse $\psi^{-1}\colon\mathbb{Y}\rightarrow\mathbb{X}$ is conformal and bijective. (\textit{iii}) If $\psi\colon\mathbb{X}\rightarrow\mathbb{Y}$ and $\xi\colon\mathbb{Y}\rightarrow\mathbb{E}$ are conformal mappings, then $\xi\circ\psi$ is conformal. 
For $\psi$ being analytic in $\mathbb{X}\cup\{s\in \mathbb{C}\big\lvert -s^*\in\mathbb{X}\}$, it holds that  $\overline{\psi}(s_0)^*=\psi(s_0^*)$ for  $s_0\in\mathbb{X}$, cf.~also the notation in Section \ref{sec:notation}.

Before we introduce the space $\mathcal{H}_2(\bar{\mathbb{A}}^{\mathsf{c}})$, which we subsequently use to define our optimal model reduction framework, we make the following assumption.
\begin{assumption}\label{assumption:1} Let the meromorphic function $\psi\colon\mathbb{C}\rightarrow\mathbb{C}$ be given. 
We assume $\psi\colon\mathbb{X} \rightarrow\mathbb{A}$, with $\mathbb{X} \subseteq \mathbb{C}_-$, to be bijective conformal. Let  $\tilde{\mathbb{X}}\subseteq\bar{\mathbb{X}}^{\mathsf{c}}$ be an open subset such that $\{s\in \mathbb{C}\big\lvert -s^*\in\mathbb{X}\}\subseteq\tilde{\mathbb{X}}$. We assume $\psi$ to conformally map $\tilde{\mathbb{X}}$ into $\bar{\mathbb{A}}^{\mathsf{c}}$. In addition, $\psi'$ is zero at a finite number of points in $\bar{\mathbb{X}}^{\mathsf{c}}$. In summary, we assume $\psi$ to fulfill
\begin{equation*}
	\begin{cases}
		\psi\colon\mathbb{X} \rightarrow\mathbb{A} \;\text{is bijective conformal},\\
		\psi\colon\tilde{\mathbb{X}} \rightarrow\bar{\mathbb{A}}^{\mathsf{c}} \;\text{is conformal},
	\end{cases}
\end{equation*}
\end{assumption}

\begin{figure}
	\centering
	\includegraphics[width=1\textwidth]{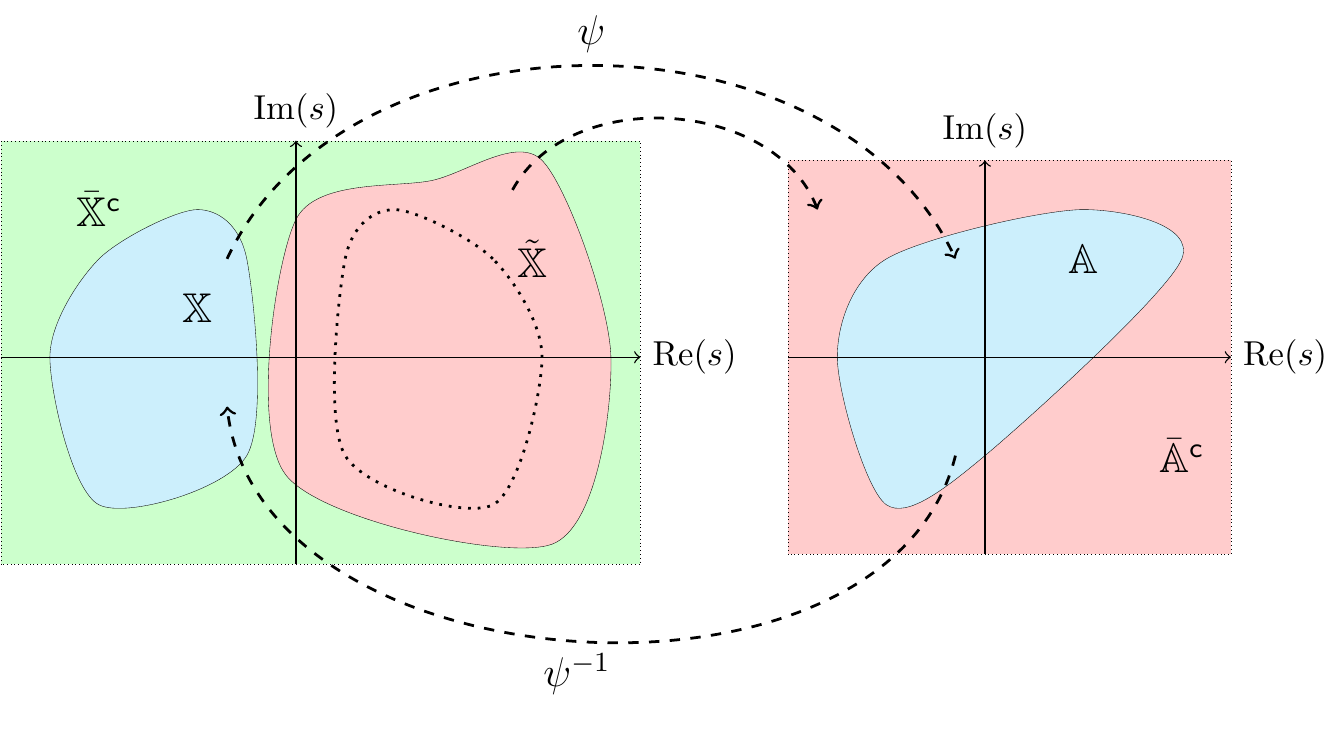}
	\caption{Depiction of the sets introduced in Assumption \ref{assumption:1} and the mapping $\psi$ along with its inverse $\psi^{-1}$. Here the dashed closed line indicates the boundary of the set $\{s\in \mathbb{C}\big\lvert -s^*\in\mathbb{X}\}$.}
	\label{fig:depictionpsi}
\end{figure}
From now on, we assume that Assumption \ref{assumption:1} is satisfied. In Figure \ref{fig:depictionpsi} we depict the domains $\mathbb{X}$, $\mathbb{A}$, $\tilde{\mathbb{X}}$ and their exteriors, along with the mapping $\psi$. 
It is worth mentioning that, for $\mathbb{A}\subset\mathbb{C}$ being a non-empty open simply connected domain and $\mathbb{X}$ being the left half complex plane $\mathbb{C}_-$, then there exists a bijective conformal map $\psi\colon\mathbb{C}_-\rightarrow\mathbb{A}$. This is a result of the Riemann mapping theorem (see \cite[Theorem 6.4.1]{Wegert2012}).
Now, consider $F$ as in (\ref{eq:functionbasic}) having poles $\lambda_{i}\in\mathbb{A}$, $i=1,\dots,n$, such that it is analytic in $\bar{\mathbb{A}}^{\mathsf{c}}$. 
It is easy to prove that in case $\psi\colon\mathbb{C}_+\rightarrow\bar{\mathbb{A}}^{\mathsf{c}}$ is analytic, then $F(\psi(\cdot))$ and $F(\psi(\cdot^*))^*$ are analytic in $\mathbb{C}_+$ (see also the proof of \cite[Theorem 6.6.2]{Wegert2012}). 
In addition, $F(\psi(-\cdot))$ and $\overline{F}(\overline{\psi}(-\cdot))=F(\psi(-\cdot^*))^*$ are analytic in $\mathbb{C}_-$. 
Similar arguments apply in the case that $\psi$ is analytic in $\mathbb{X}\subset\mathbb{C}_-$. This would then result in  $F(\psi(-\cdot^*))^*$ being analytic in the set $\{s\in \mathbb{C}\big\lvert -s^*\in\mathbb{X}\}$.

We now define the $\mathcal{H}_2$ space for functions analytic in $\bar{\mathbb{A}}^{\mathsf{c}}$. This definition relies on the concepts introduced by Duren in \cite[Chapter 10]{duren1970}. The main difference is the set in which $\psi$ is conformal.  
\begin{definition}[$\mathcal{H}_2(\bar{\mathbb{A}}^{\mathsf{c}})$ space] \label{def:H2A}
	 Let $ f\colon\bar{\mathbb{A}}^{\mathsf{c}}\rightarrow\mathbb{C}$ and $  g\colon\bar{\mathbb{A}}^{\mathsf{c}}\rightarrow\mathbb{C}$ be analytic. 
	 Denote by
	 \begin{equation} \label{eq:operatorH}
	 	\mathfrak{H}_f(s) = f(\psi(s))\psi'(s)^{\frac{1}{2}},	
	 \end{equation} 
	 then the $\mathcal{H}_2(\bar{\mathbb{A}}^{\mathsf{c}})$ inner product is defined as 
	 \begin{equation*}
	 	\left\langle  f,  g\right\rangle_{\mathcal{H}_2(\bar{\mathbb{A}}^{\mathsf{c}})} := \left\langle \mathfrak{H}_f,\mathfrak{H}_g\right\rangle_{\mathcal{H}_2(\mathbb{C}_+)}
	 \end{equation*}
	with the induced $\mathcal{H}_2(\bar{\mathbb{A}}^{\mathsf{c}})$-norm 
	\begin{equation*}
		\| f\|_{\mathcal{H}_2(\bar{\mathbb{A}}^{\mathsf{c}})} := \left\|\mathfrak{H}_f\right\|_{\mathcal{H}_2(\mathbb{C}_+)}=\left( \left\langle \mathfrak{H}_f,\mathfrak{H}_f\right\rangle_{\mathcal{H}_2(\mathbb{C}_+)}\right)^{\frac{1}{2}}.
	\end{equation*}
	The space $\mathcal{H}_2(\bar{\mathbb{A}}^{\mathsf{c}})$  is defined as 
	\begin{equation*}
		\mathcal{H}_2(\bar{\mathbb{A}}^{\mathsf{c}}):=\left\{ f\colon\bar{\mathbb{A}}^{\mathsf{c}}\rightarrow\mathbb{C}\;\textnormal{analytic}\; \bigg\lvert \| f\|_{\mathcal{H}_2(\bar{\mathbb{A}}^{\mathsf{c}})}<\infty\right\}.
	\end{equation*}
\end{definition}
It follows from Definition \ref{def:H2A} that if $f\in\mathcal{H}_2(\bar{\mathbb{A}}^{\mathsf{c}})$ then $\mathfrak{H}_f\in\mathcal{H}_2(\mathbb{C}_+)$. Note that the norm strongly depends on the set $\mathbb{A}$ as well as the chosen map $\psi$.
Obviously, for $\psi(s)=s$ we have $\mathcal{H}_2(\bar{\mathbb{A}}^{\mathsf{c}}) \equiv \mathcal{H}_2(\mathbb{C}_+)$ leading to the results mentioned in Section \ref{sec:H2}. 
Instead of computing the $\mathcal{H}_2(\bar{\mathbb{A}}^{\mathsf{c}})$ inner product through the integral in (\ref{eq:H2Cplus}), we now derive a pole residue expression similar to \cite[Lemma 2.4]{gugercin2008}. 
\begin{lemma}\label{lemma:frakHinnerproduct}
	Consider the functions $f,g\in\mathcal{H}_2(\bar{\mathbb{A}}^{\mathsf{c}})$. 
	Let $\mathfrak{H}_f$ and $\mathfrak{H}_g$ have finitely many poles $\{\mathfrak{l}_i\}_{i=1}^n\in\mathbb{C}_-$ and $\{\mathfrak{m}_j\}_{j=1}^q\in\mathbb{C}_-$. Then 
	\begin{equation*}
			\left\langle  f,  g\right\rangle_{\mathcal{H}_2(\bar{\mathbb{A}}^{\mathsf{c}})} = \sum^q_{j=1}\mathsf{res}\left[\overline{\mathfrak{H}}_f(-s)\mathfrak{H}_g(s), \mathfrak{m}_j\right] = \sum^n_{i=1}\mathsf{res}\left[\overline{\mathfrak{H}}_g(-s)\mathfrak{H}_f(s), \mathfrak{l}_i\right]^*,
	\end{equation*}
where $\overline{\mathfrak{H}}_f(-s)= \overline{f}(\overline{\psi}(-s))\overline{\psi'}(-s)^{\frac{1}{2}}$. 
\end{lemma}
\begin{proof} The proof directly follows from \cite[Lemma 2.4]{gugercin2008} applied to $\mathfrak{H}_f$ and $\mathfrak{H}_g$, respectively. 
\end{proof}
Lemma \ref{lemma:frakHinnerproduct} allows to compute the $\mathcal{H}_2(\bar{\mathbb{A}}^{\mathsf{c}})$-norm as follows
\begin{equation*}
		\|  f\|_{\mathcal{H}_2(\bar{\mathbb{A}}^{\mathsf{c}})}^2 = \sum^n_{i=1}\overline{\mathfrak{H}}_f(-\mathfrak{l}_i)\mathsf{res}\left[\mathfrak{H}_f(s), \mathfrak{l}_i\right].
\end{equation*}
We now consider the specific case where $f$ is a rational function as in (\ref{eq:functionbasic}) and therefore adopt the notation $F$ that we previously introduced for functions that have a rational structure. We begin by examining the case where $\psi$ and $\psi'(\cdot)^{\frac{1}{2}}$ share the same poles in the complex plane which happens to be the case if $\psi$ is a rational function with simple poles. If $F$ is rational as in (\ref{eq:functionbasic}), then this leads to $\mathfrak{H}_F$ having the same poles as $F\circ\psi$. In other words, if we denote by $\{\lambda_i\}_{i=1}^n$ the poles of $F$ and by $\{\mathfrak{l}_i\}_{i=1}^n$ the poles of $\mathfrak{H}_F$, then $\{\psi^{-1}(\lambda_i)\}_{i=1}^n=\{\mathfrak{l}_i\}_{i=1}^n$.
\begin{lemma}\label{lemma:residue}
	Let $F(\cdot) = \sum_{i=1}^n \frac{\phi_i}{\cdot-\lambda_i}$,
	with $\lambda_i\in\mathbb{A}$ for $i=1,\dots,n$. Let $\psi$ and $\psi'(\cdot)^{\frac{1}{2}}$ have the same poles. Then
	\begin{equation}\label{eq:residueHfrakF}
		\mathsf{res}\left[ \mathfrak{H}_F(s),\psi^{-1}(\lambda_i)\right] = \phi_i \psi'(\psi^{-1}(\lambda_i))^{-\frac{1}{2}},
	\end{equation}
	for $i=1,\dots,n$.
\end{lemma}
\begin{proof}
	 Note that since $F$ is a rational function with poles in $\mathbb{A}$, we have that $F\in\mathcal{H}_2(\bar{\mathbb{A}}^{\mathsf{c}})$. By expanding the left hand side of (\ref{eq:residueHfrakF}) and using (\ref{eq:functionbasic}), we get
	\begin{equation*}
		\begin{aligned}
			\mathsf{res}\left[ \mathfrak{H}_F(s),\psi^{-1}(\lambda_i)\right]  
			&=\mathsf{res}\left[ F(\psi(s))\psi'(s)^{\frac{1}{2}},\psi^{-1}(\lambda_i)\right] \\
			&=\lim_{s\rightarrow\psi^{-1}(\lambda_i)}(s-\psi^{-1}(\lambda_i)) \sum_{j=1}^n \frac{\phi_j}{\psi(s)-\lambda_j} \psi'(s)^{\frac{1}{2}}\\
			&= 
			\phi_i\psi'(\psi^{-1}(\lambda_i))^{\frac{1}{2}}\lim_{s\rightarrow\psi^{-1}(\lambda_i)} \frac{s-\psi^{-1}(\lambda_i)}{\psi(s)-\lambda_i}.
		\end{aligned}
	\end{equation*}
	Since both the numerator and denominator approach zero, we apply  L'H\^{o}pital's rule \cite[Chapter 9]{bradley2015} resulting in
	\begin{equation}\label{eq:hospital}
		\begin{aligned}
			\lim_{s\rightarrow\psi^{-1}(\mu)}\frac{s-\psi^{-1}(\mu)}{\psi(s)-\mu} &=  \lim_{s\rightarrow\psi^{-1}(\mu)}\frac{1}{\psi'(s)}
			= \frac{1}{\psi'(\psi^{-1}(\mu))}.
		\end{aligned}
	\end{equation}
	This leads to
	\begin{equation*}
		\begin{aligned}
			\mathsf{res}\left[ \mathfrak{H}_F(s),\psi^{-1}(\lambda_i)\right] &=\frac{\phi_i\psi'(\psi^{-1}(\lambda_i))^{\frac{1}{2}}}{\psi'(\psi^{-1}(\lambda_i))}
			=\phi_i \psi'(\psi^{-1}(\lambda_i))^{-\frac{1}{2}}.
		\end{aligned}
	\end{equation*}
\end{proof}
Combining the previous result with Lemma \ref{lemma:frakHinnerproduct} provides a simpler characterization of the $\mathcal{H}_2(\bar{\mathbb{A}}^{\mathsf{c}})$ inner product.
\begin{corollary} \label{corollary:innerproduct}
	Let $F(\cdot) = \sum_{i=1}^n \frac{\phi_i}{\cdot-\lambda_i}$ and $G(\cdot) = \sum_{j=1}^q \frac{\nu_j}{\cdot-\mu_j}$ have simple poles $\{\lambda_i\}_{i=1}^n\in\mathbb{A}$ and $\{\mu_j\}_{j=1}^q\in\mathbb{A}$, respectively. Assume that  
	$\mathfrak{H}_F$ and  $F\circ\psi$ as well as $\mathfrak{H}_G$ and $G\circ\psi$ have the same poles. Then
	\begin{equation*}
		\begin{aligned}
			\left\langle  F,  G\right\rangle_{\mathcal{H}_2(\bar{\mathbb{A}}^{\mathsf{c}})} &= \sum_{j=1}^q  \overline{\mathfrak{H}}_F(-\psi^{-1}(\mu_j))\left[  \nu_j \psi'(\psi^{-1}(\mu_j))^{-\frac{1}{2}}\right].
		\end{aligned}
	\end{equation*} 
\end{corollary}

For $F\in\mathcal{H}_2(\bar{\mathbb{A}}^{\mathsf{c}})$ and $\psi'(\cdot)^{\frac{1}{2}}$ with simple poles $\{\gamma_\ell\}_{\ell=1}^{m}\in\mathbb{C}_-$, we define  
\begin{equation}\label{eq:mathfrakf}
	\mathfrak{F}_{F}(s):=\overline{\mathfrak{H}}_F(-\psi^{-1}(s))\psi'(\psi^{-1}(s))^{-\frac{1}{2}}+\sum_{\ell=1}^{m}\overline{\mathfrak{H}}_F(-\gamma_\ell)\frac{\mathsf{res}\left[\psi'(s)^{\frac{1}{2}},\gamma_\ell\right]}{\psi(\gamma_\ell)-s}.
\end{equation}
A similar function has been introduced in \cite{anic2013} for a frequency-weighted $\mathcal{H}_2$ model reduction problem. In particular, as in \cite[Corollary 4]{anic2013}, the evaluation of $\mathfrak{F}_{F}$ and its derivative in $\mu$ can be expressed as a specific $\mathcal{H}_2(\bar{\mathbb{A}}^{\mathsf{c}})$ inner product involving rational functions of degree one and two, respectively. 
\begin{lemma}\label{lemma:frakf}
	Let $ F\in\mathcal{H}_2(\bar{\mathbb{A}}^{\mathsf{c}})$ be as in Lemma \ref{lemma:residue} and $\mu\in\mathbb{A}$. Consider $\psi$ and $\psi'(\cdot)^{\frac{1}{2}}$ with different poles. Let $\{\gamma_\ell\}_{\ell=1}^{m}\in\mathbb{C}_-$ with $\gamma_\ell\neq \psi^{-1}(\mu)$, $\ell=1,\dots,m$, be the poles of $\psi'(\cdot)^{\frac{1}{2}}$. Then 
	\begin{equation}\label{eq:frakf}
			\left\langle  F, \frac{1}{\cdot-\mu} \right\rangle_{\mathcal{H}_2(\bar{\mathbb{A}}^{\mathsf{c}})} 
			=    \mathfrak{F}_{F}(\mu), \quad \text{and } \left\langle  F, \frac{1}{(\cdot-\mu)^2} \right\rangle_{\mathcal{H}_2(\bar{\mathbb{A}}^{\mathsf{c}})} 
			=  \mathfrak{F}_{F}'(\mu).
	\end{equation}
\end{lemma}
\begin{proof}
	Once again, recall that since $\mu\in\mathbb{A}$, it holds that $\frac{1}{\cdot-\mu}\in\mathcal{H}_2(\bar{\mathbb{A}}^{\mathsf{c}})$ and $\frac{1}{(\cdot-\mu)^2}\in\mathcal{H}_2(\bar{\mathbb{A}}^{\mathsf{c}})$.
	Using Lemma \ref{lemma:frakHinnerproduct}, for the first inner product we obtain
	\begin{equation*}
		\begin{aligned}
				&\left\langle  F,\frac{1}{\cdot-\mu}\right\rangle_{\mathcal{H}_2(\bar{\mathbb{A}}^{\mathsf{c}})}\\
				&\quad\quad= \mathsf{res}\left[\overline{\mathfrak{H}}_F(-s)\frac{\psi'(s)^{\frac{1}{2}}}{\psi(s)-\mu},\psi^{-1}(\mu)\right] + \sum_{\ell=1}^m\mathsf{res}\left[\overline{\mathfrak{H}}_F(-s)\frac{\psi'(s)^{\frac{1}{2}}}{\psi(s)-\mu},\gamma_\ell\right].
		\end{aligned}
	\end{equation*}
	Since $\frac{1}{\cdot-\mu}$ is a rational function, we can apply Lemma \ref{lemma:residue} resulting in
	\begin{equation}\label{eq:intermediateresidue}
		\mathsf{res}\left[ \frac{1}{s-\mu},\psi^{-1}(\mu)\right] = \psi'(\psi^{-1}(\mu))^{-\frac{1}{2}}.
	\end{equation}
	It is worth noting that $F\in\mathcal{H}_2(\bar{\mathbb{A}}^\mathsf{c})$ implies that $\overline{\mathfrak{H}}_F(-\cdot)$ is analytic in $\mathbb{C}_-$. Using (\ref{eq:intermediateresidue}), we now obtain the first assertion in (\ref{eq:frakf}) since
	\begin{equation*}
		\begin{aligned}
			&\left\langle  F,\frac{1}{\cdot-\mu}\right\rangle_{\mathcal{H}_2(\bar{\mathbb{A}}^{\mathsf{c}})}\\
			&\quad\quad= \overline{\mathfrak{H}}_F(-\psi^{-1}(\mu))\mathsf{res}\left[\frac{\psi'(s)^{\frac{1}{2}}}{\psi(s)-\mu},\psi^{-1}(\mu)\right] + \sum_{\ell=1}^{m}\overline{\mathfrak{H}}_F(-\gamma_\ell)\frac{\mathsf{res}\left[\psi'(s)^{\frac{1}{2}},\gamma_\ell\right]}{\psi(\gamma_\ell)-s}\\
			&\quad\quad= \overline{\mathfrak{H}}_F(-\psi^{-1}(\mu))\psi'(\psi^{-1}(\mu))^{-\frac{1}{2}} + \sum_{\ell=1}^{m}\overline{\mathfrak{H}}_F(-\gamma_\ell)\frac{\mathsf{res}\left[\psi'(s)^{\frac{1}{2}},\gamma_\ell\right]}{\psi(\gamma_\ell)-s}= \mathfrak{F}_{F}(\mu).
		\end{aligned}
	\end{equation*}
	From $\psi$ being bijective conformal in $\mathbb{X}$, with the implicit function theorem (see \cite[Theorem I.5.7]{freitag2006}), we conclude that
	\begin{equation*}
		\begin{aligned}
			\frac{\mathrm{d}}{\mathrm{d}s}\psi(\psi^{-1}(s))&=\psi'(\psi^{-1}(s))\left[\psi^{-1}\right]'(s)=1,
		\end{aligned}
	\end{equation*}
	leading to 
	\begin{equation}\label{eq:idpsi}
		\begin{aligned}
			\left[\psi^{-1}\right]'(s)&=\frac{1}{\psi'(\psi^{-1}(s))}.
		\end{aligned}
	\end{equation}
	Moreover, we also know that $\psi^{-1}$ is analytic on the image of $\psi$. 	For the second inner product we have that the poles of $\psi'(\cdot)^{\frac{1}{2}}$ differ from $\psi^{-1}(\mu)$. In addition, because $\overline{\mathfrak{H}}_F(-\cdot)$ is analytic in $\mathbb{C}_-$, the term
	$\overline{\mathfrak{H}}_F(-\cdot)\psi'(\cdot)^{\frac{1}{2}}(\psi(\cdot)-\mu)^{-2}$
	has only a double pole in $\psi^{-1}(\mu)$.
	Hence, for the inner product it follows
	\begin{align}\label{eq:innerproduct2}
			&\left\langle  F,\frac{1}{(\cdot-\mu)^2}\right\rangle_{\mathcal{H}_2(\bar{\mathbb{A}}^{\mathsf{c}})} \nonumber \\ 
			&\quad\quad= \mathsf{res}\left[ \frac{\overline{\mathfrak{H}}_F(-s)}{(\psi(s)-\mu)^2} \psi'(s)^{\frac{1}{2}}, \psi^{-1}(\mu)\right] +\sum_{\ell=1}^{m}\overline{\mathfrak{H}}_F(-\gamma_\ell)\frac{\mathsf{res}\left[\psi'(s)^{\frac{1}{2}},\gamma_\ell\right]}{(\psi(\gamma_\ell)-\mu)^2} \displaybreak[0]  \nonumber \\  
			&\quad\quad= \lim_{s_0\rightarrow \psi^{-1}(\mu)}
			\left[\frac{\mathrm{d}}{\mathrm{d}s}\left(
			 (s-\psi^{-1}(\mu))^2
			\frac{\overline{\mathfrak{H}}_F(-s)}{(\psi(s)-\mu)^2}
			\psi'(s)^{\frac{1}{2}}\right) \right]_{s=s_0} \\
			&\quad\quad\quad\quad\quad +\sum_{\ell=1}^{m}\overline{\mathfrak{H}}_F(-\gamma_\ell)\frac{\mathsf{res}\left[\psi'(s)^{\frac{1}{2}},\gamma_\ell\right]}{(\psi(\gamma_\ell)-\mu)^2} \displaybreak[0] \nonumber \\
			&\quad\quad= \lim_{s_0\rightarrow\psi^{-1}(\mu)}\left[-\overline{\mathfrak{H}}_F'(-s)\left(\frac{s-\psi^{-1}(\mu)}{\psi(s)-\mu}\right)^2 \psi'(s)^{\frac{1}{2}}\right. \nonumber \\
			&\quad\quad\quad\quad\quad 
			+\left.\overline{\mathfrak{H}}_F(-s)\left(\frac{s-\psi^{-1}(\mu)}{\psi(s)-\mu}\right)^2 \frac{\psi'(s)^{-\frac{1}{2}}}{2}\psi''(s)\right. \nonumber \\
			&\quad\quad\quad\quad\quad 
			+ \left.  \overline{\mathfrak{H}}_F(-s)\frac{2\left(s-\psi^{-1}(\mu)\right)}{\psi(s)-\mu}  \right. \nonumber\\
			&\quad\quad \quad \quad \quad
			\left.
			\left(\frac{\psi^{-1}(\mu)-s}{(\psi(s)-\mu)^2}\psi'(s)+\frac{1}{\psi(s)-\mu}\right) \psi'(s)^{\frac{1}{2}}
			\right]_{s=s_0} \nonumber \\
			&\quad\quad\quad\quad\quad +\sum_{\ell=1}^{m}\overline{\mathfrak{H}}_F(-\gamma_\ell)\frac{\mathsf{res}\left[\psi'(s)^{\frac{1}{2}},\gamma_\ell\right]}{(\psi(\gamma_\ell)-\mu)^2}. \nonumber
	\end{align}
	We first focus on the terms 
	\begin{equation}\label{eq:limit1}
		\lim_{s\rightarrow\psi^{-1}(\mu)}\frac{s-\psi^{-1}(\mu)}{\psi(s)-\mu},
	\end{equation}
	and 
	\begin{equation}\label{eq:limit2}
		\lim_{s\rightarrow\psi^{-1}(\mu)}2\frac{s-\psi^{-1}(\mu)}{\psi(s)-\mu}\left(\frac{\psi^{-1}(\mu)-s}{(\psi(s)-\mu)^2}\psi'(s)+\frac{1}{\psi(s)-\mu}\right). 
	\end{equation}
	Again, as in (\ref{eq:hospital}) and (\ref{eq:idpsi}) with L'H\^{o}pital's rule we arrive at
	\begin{equation}\label{eq:hopital1}
		\begin{aligned}
			\lim_{s\rightarrow\psi^{-1}(\mu)}\frac{s-\psi^{-1}(\mu)}{\psi(s)-\mu} &= \left[\psi^{-1}\right]'(\mu).
		\end{aligned}
	\end{equation}
	Similarly, for (\ref{eq:limit2}) we find
	\begin{equation*}
		\begin{aligned}
			&\left[\psi^{-1}\right]'(\mu)\lim_{s\rightarrow\psi^{-1}(\mu)}2\left(\frac{\psi^{-1}(\mu)-s}{(\psi(s)-\mu)^2}\psi'(s)+\frac{1}{\psi(s)-\mu}\right) \\
			&\quad =\left[\psi^{-1}\right]'(\mu)\lim_{s\rightarrow\psi^{-1}(\mu)}\left(\frac{\psi''(s)}{\psi'(s)}\frac{\psi^{-1}(\mu)-s}{\psi(s)-\mu}\right) \\
			&\quad=-\left(\left[\psi^{-1}\right]'(\mu)\right)^2\left(\frac{\psi''(\psi^{-1}(\mu))}{\psi'(\psi^{-1}(\mu))}\right)
		\end{aligned}
	\end{equation*}
	Using these equalities in (\ref{eq:innerproduct2}) gives
	\begin{align*}
			&\left\langle  F,\frac{1}{(\cdot-\mu)^2}\right\rangle_{\mathcal{H}_2(\bar{\mathbb{A}}^{\mathsf{c}})}\\ 
			&\quad\quad= -\overline{\mathfrak{H}}_F'(-\psi^{-1}(\mu))\left(\left[\psi^{-1}\right]'(\mu)\right)^2\psi'(\psi^{-1}(\mu))^{\frac{1}{2}} \\
			&\quad\quad\quad\quad\quad 
			+   \overline{\mathfrak{H}}_F(-\psi^{-1}(\mu))\left(\left[\psi^{-1}\right]'(\mu)\right)^2\frac{\psi'(\psi^{-1}(\mu))^{-\frac{1}{2}}}{2}\psi''(\psi^{-1}(\mu))\\
			&\quad\quad\quad\quad\quad 
			-\overline{\mathfrak{H}}_F(-\psi^{-1}(\mu))    
			\left(\left[\psi^{-1}\right]'(\mu)\right)^2\frac{\psi''(\psi^{-1}(\mu))}{\psi'(\psi^{-1}(\mu))}\psi'(\psi^{-1}(\mu))^{\frac{1}{2}}\\
			&\quad\quad\quad\quad\quad +\sum_{\ell=1}^{m}\overline{\mathfrak{H}}_F(-\gamma_\ell)\frac{\mathsf{res}\left[\psi'(s)^{\frac{1}{2}},\gamma_\ell\right]}{(\psi(\gamma_\ell)-\mu)^2}\\
			&\quad\quad= -\overline{\mathfrak{H}}_F'(-\psi^{-1}(\mu))\left(\left[\psi^{-1}\right]'(\mu)\right)^2\psi'(\psi^{-1}(\mu))^{\frac{1}{2}} \\
			&\quad\quad\quad\quad\quad 
			-\frac{1}{2}\overline{\mathfrak{H}}_F(-\psi^{-1}(\mu))\left(\left[\psi^{-1}\right]'(\mu)\right)^2\psi'(\psi^{-1}(\mu))^{-\frac{1}{2}}\psi''(\psi^{-1}(\mu))\\
			&\quad\quad\quad\quad\quad +\sum_{\ell=1}^{m}\overline{\mathfrak{H}}_F(-\gamma_\ell)\frac{\mathsf{res}\left[\psi'(s)^{\frac{1}{2}},\gamma_\ell\right]}{(\psi(\gamma_\ell)-\mu)^2}.
	\end{align*}
Using (\ref{eq:hopital1}) eventually leads to the second equality of (\ref{eq:frakf})
\begin{align*}
		&\left\langle  F,\frac{1}{(\cdot-\mu)^2}\right\rangle_{\mathcal{H}_2(\bar{\mathbb{A}}^{\mathsf{c}})}\\ 
		&\quad\quad= -\overline{\mathfrak{H}}_F'(-\psi^{-1}(\mu))\left[\psi^{-1}\right]'(\mu)\psi'(\psi^{-1}(\mu))^{-\frac{1}{2}}\\
		&\quad\quad\quad\quad\quad 
		-\frac{1}{2}   \overline{\mathfrak{H}}_F(-\psi^{-1}(\mu))\psi'(\psi^{-1}(\mu))^{-\frac{3}{2}}\psi''(\psi^{-1}(\mu))\left[\psi^{-1}\right]'(\mu)\\
		&\quad\quad\quad\quad\quad +\sum_{\ell=1}^m\overline{\mathfrak{H}}_F(-\gamma_\ell)\frac{\mathsf{res}\left[\psi'(s)^{\frac{1}{2}},\gamma_\ell\right]}{(\psi(\gamma_\ell)-\mu)^2}= \mathfrak{F}_{F}'(\mu).
	\end{align*}
\end{proof}
It is worth noting that if $\{\psi(\gamma_\ell)\}_{\ell=1}^{m}\in\bar{\mathbb{A}}^{\mathsf{c}}$ and $\psi'(\cdot)^{\frac{1}{2}}$ is analytic in $\mathbb{X}$, 
then $\mathfrak{F}_F$ is analytic in $\mathbb{A}$. This is due to both summands in (\ref{eq:mathfrakf}) being analytic in $\mathbb{A}$.
Let us now consider the case where $\mathfrak{H}_F$ has the same poles as $F\circ\psi$ as in Corollary \ref{corollary:innerproduct}. Then we have a simplification of Lemma \ref{lemma:frakf}. As a matter of fact, $\mathfrak{F}_{F}(s)$ becomes
\begin{equation*}
	\mathfrak{F}_{F}(s)=\overline{\mathfrak{H}}_F(-\psi^{-1}(s))\psi'(\psi^{-1}(s))^{-\frac{1}{2}},
\end{equation*}
leading to the following corollary.
\begin{corollary}\label{corollary:practical}
	Consider $ F\in\mathcal{H}_2(\bar{\mathbb{A}}^{\mathsf{c}})$ 
	and $\mu\in\mathbb{A}$. If $\psi$ and $\psi'(\cdot)^{\frac{1}{2}}$ have the same poles, then
	\begin{equation*}
		\left\langle  F, \frac{1}{\cdot-\mu} \right\rangle_{\mathcal{H}_2(\bar{\mathbb{A}}^{\mathsf{c}})} 
		=    \overline{\mathfrak{H}}_F(-\psi^{-1}(\mu))\psi'(\psi^{-1}(\mu))^{-\frac{1}{2}},
	\end{equation*}
and 
	\begin{equation*}
		\left\langle  F, \frac{1}{(\cdot-\mu)^2} \right\rangle_{\mathcal{H}_2(\bar{\mathbb{A}}^{\mathsf{c}})} 
		=  \left[\frac{\mathrm{d}}{\mathrm{d}s}\overline{\mathfrak{H}}_F(-\psi^{-1}(s))\psi'(\psi^{-1}(s))^{-\frac{1}{2}}\right]_{s=\mu}.
	\end{equation*}
\end{corollary}

\section{Optimal \texorpdfstring{$\mathcal{H}_2(\bar{\mathbb{A}}^{\mathsf{c}})$}{TEXT} model reduction}\label{sec:optconditions}
The $\mathcal{H}_2$ optimal model order reduction framework seeks a solution for the optimization problem in (\ref{eq:H2optprob}). 
We now introduce a similar approach based on the $\mathcal{H}_2(\bar{\mathbb{A}}^{\mathsf{c}})$ space described in Definition \ref{def:H2A}. 
The objective is to find an optimal reduced order model (\ref{eq:rom}) that minimizes the error norm
\begin{equation}\label{eq:H2AbarOptProb}
\min_{\tilde{H}}\|H-\tilde{H}\|_{\mathcal{H}_2(\bar{\mathbb{A}}^{\mathsf{c}})}.
\end{equation}
By definition we can recast this optimization problem as
\begin{equation}\label{eq:aux}
	\min_{\tilde{H}}\|\mathfrak{H}_H-\mathfrak{H}_{\tilde{H}}\|_{\mathcal{H}_2(\mathbb{C}_+)}.
\end{equation}
However, note that it is not possible to perform classical $\mathcal{H}_2(\mathbb{C}_+)$ optimal model reduction for $\mathfrak{H}_H$, the reason being that the reduced model $\mathfrak{H}_{\tilde{H}}$ would have to possess a very particular structure. In fact, \eqref{eq:aux} presents a structure-preserving $\mathcal{H}_2$ model reduction problem that similarly arises for the frequency-weighted case, see, once again, \cite{anic2013}. Contrary to the latter work, under certain assumptions on $\psi,\psi'$ and $\psi^{-1}$, we will be able to derive optimality conditions which allow for a numerical approach based on a slight modification of IRKA.

Of course, the nonconvexity of (\ref{eq:H2AbarOptProb}) makes the computation of a global minimizer very challenging if not impossible. 
For this reason, one seeks \textit{local minimizers}, making the task more feasible.  Consequently, let us consider the transfer function $\widehat{H}$ of the reduced model in (\ref{eq:rom}) to be a local minimizer of (\ref{eq:H2AbarOptProb}).
The following theorem states the interpolation conditions that $\widehat{H}$ needs to fulfill. The proof follows a perturbation argument that has previously been used  in, e.g.,  \cite{gugercin2008,anic2013,antoulas2020}. In particular, consider a transfer function $\widehat{H}^{(\varepsilon)}$ with $\|\widehat{H}-\widehat{H}^{(\varepsilon)}\|_{\mathcal{H}_2(\bar{\mathbb{A}}^{\mathsf{c}})}=\mathcal{O}(\varepsilon)$ so that the local minimality of $\widehat{H}$ implies
\begin{equation}\label{eq:perturb}
	\|H-\widehat{H}\|_{\mathcal{H}_2(\bar{\mathbb{A}}^{\mathsf{c}})}\leq \|H-\widehat{H}^{(\varepsilon)}\|_{\mathcal{H}_2(\bar{\mathbb{A}}^{\mathsf{c}})}.
\end{equation}
\begin{theorem}\label{theorem:main}
	Let $\mathbb{A}$ be a non-empty connected open set and $H\in\mathcal{H}_2(\bar{\mathbb{A}}^{\mathsf{c}})$. Consider $\psi$ as in Assumption \ref{assumption:1} and let $\psi'(\cdot)^{\frac{1}{2}}$ be analytic in $\mathbb{X}$ and have the poles $\{\gamma_\ell\}_{\ell=1}^{m}$ such that $\{\psi(\gamma_\ell)\}_{\ell=1}^{m}\in\bar{\mathbb{A}}^{\mathsf{c}}$. Let $\widehat{H}\in\mathcal{H}_2(\bar{\mathbb{A}}^{\mathsf{c}})$ be a local minimizer of (\ref{eq:H2AbarOptProb}) with poles $\{\widehat{\lambda}_j\}_{j=1}^r\in\mathbb{A}$. Let $\{\psi^{-1}(\widehat{\lambda}_j)\}_{j=1}^r$ be different from the poles of $\psi'(\cdot)^{\frac{1}{2}}$ and $\psi$. Then the following interpolation conditions hold for $j=1,\dots,r$
	\begin{equation}\label{eq:interpolationcondH2A}
		\begin{aligned}
			\mathfrak{F}_{\widehat{H}}(\widehat{\lambda}_j) = \mathfrak{F}_{H}(\widehat{\lambda}_j)\quad \text{and}\quad \mathfrak{F}_{\widehat{H}}'(\widehat{\lambda}_j) = \mathfrak{F}_{H}'(\widehat{\lambda}_j).
		\end{aligned}
	\end{equation} 
\end{theorem}
\begin{proof}	
For the first condition, we consider a perturbation of the local optimum in the $p$-th residue such that 
\begin{equation*}
	\widehat{H}^{(\varepsilon)}(s) =\frac{\widehat{\phi}_p+\varepsilon e^{\mathrm{i}\theta_1}}{s-\widehat{\lambda}_p}+ \sum_{j\neq p} \frac{\widehat{\phi}_j}{s-\widehat{\lambda}_j},
\end{equation*}	
with $\varepsilon>0$ small and $\theta_1$ arbitrary.
We then get from (\ref{eq:perturb})
\begin{equation*}
	\begin{aligned}
		\|H-\widehat{H}\|^2_{\mathcal{H}_2(\bar{\mathbb{A}}^{\mathsf{c}})}&\leq \|H-\widehat{H}^{(\varepsilon)}\|_{\mathcal{H}_2(\bar{\mathbb{A}}^{\mathsf{c}})}^2
		= \|H - \widehat{H} +\widehat{H}-\widehat{H}^{(\varepsilon)}\|_{\mathcal{H}_2(\bar{\mathbb{A}}^{\mathsf{c}})}^2\\
		&= \|H - \widehat{H}\|_{\mathcal{H}_2(\bar{\mathbb{A}}^{\mathsf{c}})}^2 +2\text{Re}\left\{\left\langle H - \widehat{H}, \widehat{H}-\widehat{H}^{(\varepsilon)} \right\rangle_{\mathcal{H}_2(\bar{\mathbb{A}}^{\mathsf{c}})}\right\}\\
		& \quad\quad\quad + \|\widehat{H}-\widehat{H}^{(\varepsilon)}\|_{\mathcal{H}_2(\bar{\mathbb{A}}^{\mathsf{c}})}^2
	\end{aligned}
\end{equation*}
which leads to
\begin{equation}\label{eq:norminequality}
	\begin{aligned}
		0&\leq 2\text{Re}\left\{\left\langle H - \widehat{H}, \widehat{H}-\widehat{H}^{(\varepsilon)} \right\rangle_{\mathcal{H}_2(\bar{\mathbb{A}}^{\mathsf{c}})}\right\}+ \|\widehat{H}-\widehat{H}^{(\varepsilon)}\|_{\mathcal{H}_2(\bar{\mathbb{A}}^{\mathsf{c}})}^2.
	\end{aligned}
\end{equation}
We use Lemma \ref{lemma:frakf} and (\ref{eq:mathfrakf}) to evaluate the two terms in (\ref{eq:norminequality}) 
\begin{equation}
	\begin{aligned}
		\left\langle H - \widehat{H}, \widehat{H}-\widehat{H}^{(\varepsilon)} \right\rangle_{\mathcal{H}_2(\bar{\mathbb{A}}^{\mathsf{c}})} &=  \left\langle H - \widehat{H}, \frac{-\varepsilon e^{\mathrm{i}\theta_1}}{s-\widehat\lambda_p} \right\rangle_{\mathcal{H}_2(\bar{\mathbb{A}}^{\mathsf{c}})} \\ 
		&= -\varepsilon e^{\mathrm{i}\theta_1} \left(\mathfrak{F}_{H}(\widehat{\lambda}_p)-\mathfrak{F}_{\widehat{H}}(\widehat{\lambda}_p) \right),
	\end{aligned}
	\label{eq:innerproduct}
\end{equation}
and
\begin{equation*}
	\begin{aligned}
		\|\widehat{H}-\widehat{H}^{(\varepsilon)}\|_{\mathcal{H}_2(\mathbb{A})}^2 &= \left\| \frac{-\varepsilon e^{\mathrm{i}\theta_1}}{s-\widehat\lambda_p} \right\|_{\mathcal{H}_2(\bar{\mathbb{A}}^{\mathsf{c}})}^2
		= \mathcal{O}(\varepsilon^2),\quad \text{for } \varepsilon\rightarrow 0.
	\end{aligned}
\end{equation*}
Consider $\theta_1$ chosen such that $e^{\mathrm{i}\theta_1} \left(\mathfrak{F}_{H}(\widehat{\lambda}_p)-\mathfrak{F}_{\widehat{H}}(\widehat{\lambda}_p) \right)$ is positive and real valued, i.e. $\theta_1 =- \mathsf{arg}\left(\mathfrak{F}_{H}(\widehat{\lambda}_p)-\mathfrak{F}_{\widehat{H}}(\widehat{\lambda}_p) \right)$.  
This implies that (\ref{eq:perturb}) becomes 
\begin{equation*}
	0\leq2 \left\lvert\mathfrak{F}_{H}(\widehat{\lambda}_p)-\mathfrak{F}_{\widehat{H}}(\widehat{\lambda}_p)\right\lvert\leq\mathcal{O}(\varepsilon).
\end{equation*}
As this holds for arbitrary $\varepsilon >0$, in the limit we obtain  
\begin{equation}\label{eq:firstoptimalityconditionsH2barAc}
	\mathfrak{F}_{H}(\widehat{\lambda}_p)=\mathfrak{F}_{\widehat{H}}(\widehat{\lambda}_p),
\end{equation}
resulting in the first part of (\ref{eq:interpolationcondH2A}).
The second interpolation condition is proven as above but with a perturbation of the $p$-th pole, i.e.,
\begin{equation*}
	\widehat{H}^{(\varepsilon)}(s) =\frac{\widehat{\phi}_p}{s-(\widehat{\lambda}_p+\varepsilon e^{\mathrm{i}\theta_2})}+ \sum_{j\neq p} \frac{\widehat{\phi}_j}{s-\widehat{\lambda}_j}.
\end{equation*}
Let us first expand the inner product
\begin{align*}
	\left\langle H - \widehat{H}, \widehat{H}-\widehat{H}^{(\varepsilon)} \right\rangle_{\mathcal{H}_2(\bar{\mathbb{A}}^{\mathsf{c}})} &= \left\langle H - \widehat{H}, \frac{\widehat{\phi}_p}{s-\widehat{\lambda}_p}- \frac{\widehat{\phi}_p}{s-(\widehat{\lambda}_p+\varepsilon e^{\mathrm{i}\theta_2})} \right\rangle_{\mathcal{H}_2(\bar{\mathbb{A}}^{\mathsf{c}})}\\
	&= \widehat{\phi}_p \left(\mathfrak{F}_{H}(\widehat{\lambda}_p)-\mathfrak{F}_{\widehat{H}}(\widehat{\lambda}_p) \right)\\ 
	&\quad\quad\quad- \widehat{\phi}_p \left(\mathfrak{F}_{H}(\widehat{\lambda}_p+\varepsilon e^{\mathrm{i}\theta_2})-\mathfrak{F}_{\widehat{H}}(\widehat{\lambda}_p+\varepsilon e^{\mathrm{i}\theta_2}) \right)\\
	&= - \widehat{\phi}_p \left(\mathfrak{F}_{H}(\widehat{\lambda}_p+\varepsilon e^{\mathrm{i}\theta_2})-\mathfrak{F}_{\widehat{H}}(\widehat{\lambda}_p+\varepsilon e^{\mathrm{i}\theta_2}) \right)
\end{align*}
where in the last equality we have used the first optimality condition (\ref{eq:firstoptimalityconditionsH2barAc}).
Because $\{\psi(\gamma_\ell)\}_{\ell=1}^{m}\in\bar{\mathbb{A}}^{\mathsf{c}}$ and $\psi'(\cdot)^{\frac{1}{2}}$ is analytic in $\mathbb{X}$, we have that both $\mathfrak{F}_H$ and $\mathfrak{F}_{\widehat{H}}$ are analytic in $\mathbb{A}$. For $\varepsilon\rightarrow0$, we consider the Taylor expansion of the two functions around $\widehat{\lambda}_p$ (see also \cite[Theorem 5.1.1]{antoulas2020}). This leads to
\begin{align*}
	\left\langle H - \widehat{H}, \widehat{H}-\widehat{H}^{(\varepsilon)} \right\rangle_{\mathcal{H}_2(\bar{\mathbb{A}}^{\mathsf{c}})} &= - \varepsilon e^{\mathrm{i}\theta_2}\widehat{\phi}_p \left(\mathfrak{F}_{H}'(\widehat{\lambda}_p)-\mathfrak{F}_{\widehat{H}}'(\widehat{\lambda}_p) \right).
\end{align*}	
For $\varepsilon\rightarrow0$ we also have that $\|\widehat{H}-\widehat{H}^{(\varepsilon)}\|_{\mathcal{H}_2(\mathbb{A})}^2=\mathcal{O}(\varepsilon^2)$. Inserting these equalities in (\ref{eq:norminequality}) results in 
\begin{equation}\label{eq:inequalityderivative}
	0\leq -2\varepsilon\text{Re}\left\{ e^{\mathrm{i}\theta_2}\widehat{\phi}_p \left(\mathfrak{F}'_{H}(\widehat{\lambda}_p)-\mathfrak{F}'_{\widehat{H}}(\widehat{\lambda}_p) \right)\right\} + \mathcal{O}(\varepsilon^2).
\end{equation}
Choosing $\theta_2$ such that $e^{\mathrm{i}\theta_2}\widehat{\phi}_p (\mathfrak{F}'_{H}(\widehat{\lambda}_p)-\mathfrak{F}'_{\widehat{H}}(\widehat{\lambda}_p))$ is positive and real valued gives us
\begin{equation*}
	\mathfrak{F}_{H}'(\widehat{\lambda}_p)=\mathfrak{F}_{\widehat{H}}'(\widehat{\lambda}_p),
\end{equation*}
obtaining the second equality of (\ref{eq:interpolationcondH2A}).
This is then repeated for $p=1,\dots,r$.
\end{proof}

Theorem \ref{theorem:main} gives us necessary optimality conditions that must hold for $\widehat{H}\in\mathcal{H}_2(\bar{\mathbb{A}}^{\mathsf{c}})$ to solve (\ref{eq:H2AbarOptProb}). 
Due to the particular structure of $\mathfrak{F}$, designing an algorithm similar to IRKA that computes a reduced model complying with \eqref{eq:interpolationcondH2A}, becomes a difficult task. Indeed, while rational approximation of transfer functions of standard LTI systems is well-known to be realizable by a Petrov-Galerkin projection onto rational Krylov subspaces, such an approach is no longer viable for the specific structure of $\mathfrak{F}$ in \eqref{eq:frakf} and it is not clear how to ensure interpolation.
Similar challenges were discussed for the weighted $\mathcal{H}_2$ model reduction problem in \cite{breiten2015,anic2013}. 
\begin{corollary}\label{corollary:practicaltheorem}
	Let the assumptions in Theorem \ref{theorem:main} hold. Let $\psi$ and $\psi'(\cdot)^{\frac{1}{2}}$ have the same poles. Assume that in neighborhoods of the points  $-\psi^{-1}(\widehat{\lambda}_j)^*$ the derivative $\overline{\psi'}$ exists and is not zero for $j=1,\dots,r$. 
	Then the optimality conditions in (\ref{eq:interpolationcondH2A}) become 
	\begin{equation}\label{eq:interpolationcondH2Apractical}
		\begin{aligned}
			\widehat{H}(\varphi(\widehat{\lambda}_j)) = H(\varphi(\widehat{\lambda}_j))\quad \text{and}\quad \widehat{H}'(\varphi(\widehat{\lambda}_j)) = H'(\varphi(\widehat{\lambda}_j)),
		\end{aligned}
	\end{equation} 
	where $\varphi(s)=\overline{\psi}(-\psi^{-1}(s))^*$.
\end{corollary}
\begin{proof} The proof directly follows using Corollary \ref{corollary:practical} in Theorem \ref{theorem:main}. Since $\psi$ and $\psi'(\cdot)^{\frac{1}{2}}$ share the same poles with Corollary \ref{corollary:practical}, we simplify (\ref{eq:interpolationcondH2A}) according to
	\begin{equation}\label{eq:interpolationsimpler}
		\overline{\mathfrak{H}}_H(-\psi^{-1}(\widehat{\lambda}_p)) = \overline{\mathfrak{H}}_{\widehat{H}}(-\psi^{-1}(\widehat{\lambda}_p))\; \text{ and } \overline{\mathfrak{H}}_H'(-\psi^{-1}(\widehat{\lambda}_p)) = \overline{\mathfrak{H}}_{\widehat{H}}'(-\psi^{-1}(\widehat{\lambda}_p)).
	\end{equation}
Using (\ref{eq:operatorH}) in the first equality of (\ref{eq:interpolationsimpler}) results in
\begin{equation}\label{eq:interpolationsimpler2}
	\widehat{H}(\varphi(\widehat{\lambda}_p))^*\overline{\psi'}(-\psi^{-1}(\widehat{\lambda}_p))^{\frac{1}{2}} = H(\varphi(\widehat{\lambda}_p))^*\overline{\psi'}(-\psi^{-1}(\widehat{\lambda}_p))^{\frac{1}{2}},
\end{equation}
where $\varphi(s)=\overline{\psi}(-\psi^{-1}(s))^*$.
The assumption made on $\overline{\psi'}$  states that $\psi'(-s^*)^*=\overline{\psi'}(-s)$ is nonzero at $s=\psi^{-1}(\widehat{\lambda}_p)$.
This allows one to simplify (\ref{eq:interpolationsimpler2}) to obtain
\begin{equation*}
		\widehat{H}(\varphi(\widehat{\lambda}_p)) = H(\varphi(\widehat{\lambda}_p)).
\end{equation*} 

Before continuing with the second equality we calculate the following derivative for $s=\psi^{-1}(\widehat{\lambda}_p)$ 
\begin{equation}\label{eq:derHfrak}
	\begin{aligned}
		\frac{\mathrm{d}}{\mathrm{d}s}\overline{\mathfrak{H}}_H(-s)&=\frac{\mathrm{d}}{\mathrm{d}s} \overline{H}(\overline{\psi}(-s))\overline{\psi'}(-s)^{\frac{1}{2}}
		=\frac{\mathrm{d}}{\mathrm{d}s} H(\overline{\psi}(-s)^*)^*\overline{\psi'}(-s)^{\frac{1}{2}}\\
		&=\frac{\mathrm{d}}{\mathrm{d}s} H(\psi(-s^*))^*\overline{\psi'}(-s)^{\frac{1}{2}}\\
		&= -H'(\psi(-s^*))^*\psi'(-s^*)^*\overline{\psi'}(-s)^{\frac{1}{2}} +H(\psi(-s^*))^*\frac{\mathrm{d}}{\mathrm{d}s}\overline{\psi'}(-s)^{\frac{1}{2}}.
	\end{aligned}
\end{equation}
Using (\ref{eq:derHfrak}) in the second equality of (\ref{eq:interpolationsimpler}) and utilizing the first interpolation conditions in (\ref{eq:interpolationcondH2Apractical}) leads to 
\begin{equation*}
	\widehat{H}'(\varphi(\widehat{\lambda}_p)) = H'(\varphi(\widehat{\lambda}_p)),
\end{equation*}
concluding the proof.
\end{proof}
We will adopt the above simplified optimality conditions to design a generalized version of IRKA in Section \ref{sec:irkacom}.

\subsection{Analysis of specific conformal mappings}

In what follows we discuss the validity of Assumption \ref{assumption:1} for some particular conformal maps. In addition, we show that these functions also satisfy the assumptions made in Corollary \ref{corollary:practicaltheorem}. These functions will subsequently be used in the numerical examples of Section \ref{sec:numexp}. 

\subsubsection{Obtaining the \texorpdfstring{$\mathcal{H}_2(\bar{\mathbb{D}}^{\mathsf{c}})$}{TEXT} optimality conditions}\label{sec:H2Doptimalconditions}
We first recover the $\mathcal{H}_2(\bar{\mathbb{D}}^{\mathsf{c}})$ optimality conditions (\ref{eq:h2Doptimalconditions}) 
as a specific case of the $\mathcal{H}_2(\bar{\mathbb{A}}^{\mathsf{c}})$ framework. The meromorphic function $\psi$ that conformally maps $\mathbb{C_-}$ into $\mathbb{D}$ is the following M\"{o}bius transformation \cite{Wegert2012}
\begin{equation}\label{eq:mobiusD}
	\psi(s) = \frac{s+1}{s-1},
\end{equation}
with inverse $\psi^{-1}(s) = \frac{s+1}{s-1}$. 
Let us note that $\psi$ also conformally maps $\mathbb{C}_+\backslash\{1\}$ into $\bar{\mathbb{D}}^{\mathsf{c}}$. Being the derivative of (\ref{eq:mobiusD})
\begin{equation}\label{eq:dmobiusD}
	\psi'(s)
	=\frac{-2}{(s-1)^2},
\end{equation}
we can see that it is not zero in the complex plane, excluding the double pole in 1.
If we consider only $\mathbb{C_-}\backslash\{-1\}$ as domain of $\psi$, instead of the entire left half plane, then the mapping (\ref{eq:mobiusD}) satisfies Assumption \ref{assumption:1}. Let us note that, because of the previews consideration, we now have that $\psi\colon\mathbb{C}_-\backslash\{-1\}\rightarrow\mathbb{D}\backslash\{0\}$ is bijective and conformal. Let us now calculate the square root of (\ref{eq:dmobiusD}) 
\begin{equation*}
	\begin{aligned}\psi'(s)^{\frac{1}{2}}
		&=\frac{\mathrm{i}\sqrt{2}}{s-1}.
	\end{aligned}
\end{equation*}
We have that $\psi'(\cdot)^\frac{1}{2}$ has the same poles as $\psi$ and is non-zero in $\mathbb{C}\backslash\{1\}$. Furthermore, for the full order transfer function 
\begin{equation}\label{eq:tffom}
	H(s) = \sum_{j=1}^n \frac{\phi_j}{s-\lambda_j}, 
\end{equation}
we have that 
\begin{align*}
	\mathfrak{H}_H(s) &= H(\psi(s))\psi'(s)^{\frac{1}{2}}  
	= \sum_{j=1}^n \frac{\mathrm{i}\sqrt{2}\phi_j}{s+1-\lambda_j(s-1)}.
\end{align*}
Hence, the poles of $\mathfrak{H}_H(s)$ are given by  $\psi^{-1}(\lambda_j)=\frac{\lambda_j+1}{\lambda_j-1}$. For $\lambda_j\in\mathbb{D}\backslash\{0\}$ and $\psi\colon\mathbb{C}_-\backslash\{-1\}\rightarrow\mathbb{D}\backslash\{0\}$ being bijective and conformal, we then have that $\psi^{-1}(\lambda_j)\in\mathbb{C}_-\backslash\{-1\}$ for $j=1,\dots,n$.  
In addition, the poles of $\mathfrak{H}_H(s)$ are the same as the ones of
\begin{equation*}
	H(\psi(s)) = \sum_{j=1}^n \frac{\phi_j(s-1)}{s+1-\lambda_j(s-1)}.
\end{equation*}
For this reason we can apply Corollary \ref{corollary:practicaltheorem} to retrieve the optimal interpolation conditions.
Knowing that $\overline{\psi}(-s) = \frac{-s+1}{-s-1}$, we first expand the function
\begin{equation}\label{eq:intpointdisk}
	\begin{aligned}
		\varphi(s) &= \overline{\psi}\left(-\psi^{-1}(s)\right)^*
		= \left(\frac{-\frac{s+1}{s-1}+1}{-\frac{s+1}{s-1}-1}\right)^*
		= \frac{1}{s^*} .
	\end{aligned}
\end{equation}

From Corollary \ref{corollary:practicaltheorem} we have that 
\begin{equation*}
	H(\varphi(\widehat{\lambda}_p))=\widehat{H}(\varphi(\widehat{\lambda}_p)) \text{  and } H'(\varphi(\widehat{\lambda}_p))=\widehat{H}'(\varphi(\widehat{\lambda}_p)).
\end{equation*}
Plugging in (\ref{eq:intpointdisk}) gives us the optimal interpolation conditions in (\ref{eq:h2Doptimalconditions})
\begin{equation*}
	H\left(1/\widehat{\lambda}_p^*\right)=\widehat{H}\left(1/\widehat{\lambda}_p^*\right)\text{ and } H'\left(1/\widehat{\lambda}_p^*\right)=\widehat{H}'\left(1/\widehat{\lambda}_p^*\right).
\end{equation*}

It is worth mentioning that there exists a slight discrepancy between the above framework and the original $\mathcal{H}_2(\bar{\mathbb D}^{\mathsf{c}})$ formulation for discrete time systems. This is due to the restriction of the poles to be in $\mathbb{D}\backslash\{0\}$ such that we avoid the singularity of $\psi$ in (\ref{eq:mobiusD}). 

\subsubsection{Optimality conditions for the upper half complex plane}\label{sec:upperhalfplane}
We now study the case where the full order transfer function in (\ref{eq:tffom}) has poles on the upper half complex plane $\mathbb{C}_{\uparrow}=\left\{z\in\mathbb{C}\big\lvert \text{Im}(z)>0 \right\}$. In this case, we choose the conformal map $\psi(s)=-\mathrm{i}s$ with inverse $\psi^{-1}(s)=\mathrm{i}s$. The function is analytic in the entire complex plane and conformally maps $\mathbb{C_-}$ into $\mathbb{C}_{\uparrow}$ and $\mathbb{C_+}$ into $\mathbb{C}_{\downarrow}=\left\{z\in\mathbb{C}\big\lvert \text{Im}(z)<0 \right\}$. For this reason, Assumption \ref{algoirka} is satisfied. In addition, neither $\psi$ nor $\psi'$ have poles, they are both analytic in $\mathbb{C}$, and $\psi'$ is nonzero everywhere. Thanks to these properties we satisfy the assumptions of Corollary \ref{corollary:practicaltheorem}. By expanding the function 
\begin{equation}\label{eq:upperhalfplaneintpoints}
	\varphi(s) = \overline{\psi}(-\psi^{-1}(s))^*
	=  \overline{\psi}(-\mathrm{i}s)^* = s^*, 
\end{equation}
we then get the interpolation conditions
\begin{equation*}
	H\left(\widehat{\lambda}_p^*\right)=\widehat{H}\left(\widehat{\lambda}_p^*\right)\text{ and } H'\left(\widehat{\lambda}_p^*\right)=\widehat{H}'\left(\widehat{\lambda}_p^*\right).
\end{equation*}
In other words, the interpolation points mirror the poles of $\widehat{H}$ with respect to the real axis. This is to be expected as we are rotating the framework of IRKA by $\frac{\pi}{2}$.

\subsubsection{Optimality conditions for an ellipse}\label{sec:ellipse}

In this last example, we consider the Bernstein ellipse $\delta\mathbb{B}$, i.e., an  ellipse with foci at $1$ and $-1$ (see \cite[Chapter 8]{trefethen2019}). We refer to the interior of $\delta\mathbb{B}$ as $\mathbb{B}$. Letting $[-1,1]$ be a slit in the real axis, we assume the full order model to have poles in $\mathbb{B}\backslash[-1,1]$.
\begin{figure}
	\centering
	\includegraphics[width=0.9\textwidth]{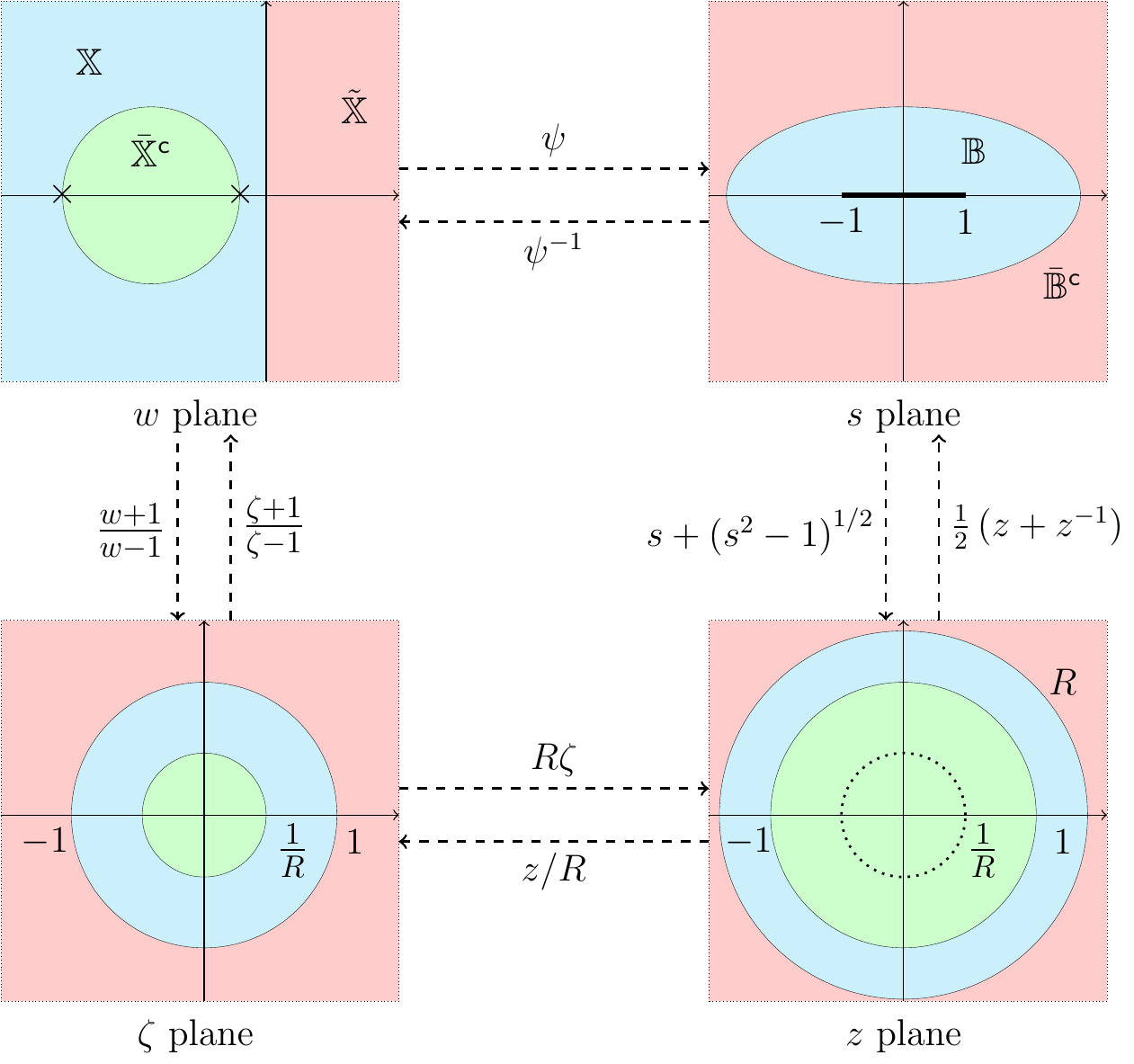}
	\caption{Mapping from the left half plane into the interior of a Bernstein ellipse with major and minor axis $(R+R^{-1})/2$ and $(R-R^{-1})/2$, respectively.}
	\label{fig:joukowski}
\end{figure}
In this case, the conformal map is designed following the process depicted in Figure \ref{fig:joukowski}. From the $w$ plane to the $s$ plane we use conformal mappings in the following order: M\"obius transformation, scaling by $R>1$, and the Joukowski transform \cite[Section 6]{kythe2019}. This leads to
\begin{equation}
	\begin{aligned}
		\psi(w) &= \frac{1}{2}\left(R\frac{w+1}{w-1}+R^{-1}\frac{w-1}{w+1}\right),\ \
		\psi^{-1}(s) = \frac{\left(s+(s^2-1)^{\frac{1}{2}}\right)R^{-1}+1}{\left(s+(s^2-1)^{\frac{1}{2}}\right)R^{-1}-1}.
	\end{aligned}
\label{eq:joukowskiextended}
\end{equation} 
Let us focus on the Joukowski transform and its inverse
\begin{equation}
	\begin{aligned}
		J(z) &= \frac{1}{2}\left(z+\frac{1}{z}\right),\quad 
		J^{-1}(s) = s\pm(s^2-1)^{\frac{1}{2}}.\\
	\end{aligned}
\label{eq:joukowski}
\end{equation}
This function maps both $\bar{\mathbb{D}}^{\mathsf{c}}$ and $\mathbb{D}\backslash\{0\}$ into $\mathbb{C}\backslash[-1,1]$. 
In particular, it maps circles of radius $R$ and $1/R$, with $R>1$, into Bernstein ellipses with major axis $(R+R^{-1})/2$ and minor axis $(R-R^{-1})/2$. Since we have $J(z)=J(1/z)$, the Joukowski transform is obviously not bijective and, as a remedy, we choose the positive root of $J^{-1}$ and the exterior of $\mathbb{D}$ as domain of $J$. 
In the $z$ plane of Figure \ref{fig:joukowski}, the domain of $J$ is represented as the blue torus and the red plane outside the disk of radius $R$. In particular, the blue torus is mapped into $\mathbb{B}\backslash[-1,1]$ and the red plane into $\bar{\mathbb{B}}^{\mathsf{c}}$. 

We now analyze the conformal map (\ref{eq:joukowskiextended}) in more detail. This meromorphic function in $\mathbb{C}$ has two poles in $1$ and $-1$. Its derivative
\begin{align*}
	\psi'(w) &= \frac{-R}{(w-1)^2} + \frac{R^{-1}}{(w+1)^2}
	= \frac{R^{-1}(w-1)^2-R(w+1)^2}{(w-1)^2(w+1)^2},
\end{align*}
exists everywhere in $\mathbb{C}\backslash\{-1,1\}$ and has two double poles in $1$ and $-1$. It presents two zeroes in $(\pm R^{-1}+1)/(\pm R^{-1}-1)$ and a double one at infinity. Since we choose $R>1$, the two zeroes are in the left half $w$ plane. In particular, these are marked with a cross in Figure \ref{fig:joukowski} along the boundary of $\bar{\mathbb{X}}^{\mathsf{c}}$. Function $\psi$ conformally and bijectively maps $\mathbb{X}$ into $\mathbb{B}\backslash[-1,1]$, and $\tilde{\mathbb{X}}=\mathbb{C}_+\backslash\{1\}$ into $\bar{\mathbb{B}}^{\mathsf{c}}$. We do not consider the disk in the $w$ plane given by $\bar{\mathbb{X}}^{\mathsf{c}}\backslash\tilde{\mathbb{X}}$ otherwise we would lose bijectivity. With these considerations we satisfy Assumption \ref{assumption:1}. In addition, we have that $\psi'(\cdot)^{\frac{1}{2}}$ and $\psi$ share the same poles, and $\psi'$ is zero in a finite amount of points in $\mathbb{C}\backslash\{-1,1\}$. As a consequence, (\ref{eq:joukowskiextended}) satisfies the assumptions of Corollary \ref{corollary:practicaltheorem}. The optimal interpolation points are then given by the function 
\begin{align*}
	\varphi(s) &= \overline{\psi}(-\psi^{-1}(s))^*= \overline{\psi}\left(-\frac{\left(s+(s^2-1)^{\frac{1}{2}}\right)R^{-1}+1}{\left(s+(s^2-1)^{\frac{1}{2}}\right)R^{-1}-1}\right)^*.
\end{align*}
Because $\overline{\psi}(\cdot)=\psi(\cdot)$ we have that
\begin{equation}\label{eq:ellipseintpoints}
	\begin{aligned}
		\varphi(s) 
		&= \psi\left(-\frac{\left(s+(s^2-1)^{\frac{1}{2}}\right)R^{-1}+1}{\left(s+(s^2-1)^{\frac{1}{2}}\right)R^{-1}-1}\right)^*\\
		&= \frac{1}{2}\left(\frac{R^2}{s+( 			 s^2-1)^{\frac{1}{2}}}+\frac{s+(s^2-1)^{\frac{1}{2}}}{R^2}\right)^*.
	\end{aligned}
\end{equation}
Using (\ref{eq:ellipseintpoints}) in (\ref{eq:interpolationcondH2Apractical}) we obtain optimality conditions for transfer functions with poles in $\mathbb{B}\backslash[-1,1]$.
In Section \ref{sec:numexp}, we will consider the conformal map  (\ref{eq:joukowskiextended}) with two minor modifications: 1) a translation of the ellipse by $c\in\mathbb{C}$ and 2) a scaling and rotation by $M\in\mathbb{C}$. The resulting conformal map becomes 
\begin{equation}\label{eq:joukowskiextendedcM}
	\begin{aligned}
		\psi(w) &= c+\frac{M}{2}\left(R\frac{w+1}{w-1}+R^{-1}\frac{w-1}{w+1}\right),
	\end{aligned}
\end{equation} 
with inverse
\begin{equation*}
	\begin{aligned}
		\psi^{-1}(s) &= \frac{\left((s-c)M^{-1}+\left(\left((s-c)M^{-1}\right)^2-1\right)^{\frac{1}{2}}\right)R^{-1}+1}{\left((s-c)M^{-1}+\left(\left((s-c)M^{-1}\right)^2-1\right)^{\frac{1}{2}}\right)R^{-1}-1}.
	\end{aligned}
\end{equation*} 
The interpolation points are then computed by
\begin{equation}\label{eq:ellipseintpointscM}
	\begin{aligned}
		\varphi(s) 
		&= c+\frac{M}{2}\left(\frac{R^2}{(s-c)M^{-1}+\left( 			 \left((s-c)M^{-1}\right)^2-1\right)^{\frac{1}{2}}}\right.\\
		&\quad\quad\quad\quad\quad \left.+\frac{(s-c)M^{-1}+\left(\left((s-c)M^{-1}\right)^2-1\right)^{\frac{1}{2}}}{R^2}\right)^*.
	\end{aligned}
\end{equation}

\subsection{IRKA with conformal maps}\label{sec:irkacom}

In view of the interpolation conditions from Corollary \ref{corollary:practicaltheorem}, it now seems natural to consider an iterative algorithm to solve (\ref{eq:H2AbarOptProb}) by modification of IRKA. In particular, instead of updating the interpolation points according to a reflection along the imaginary axis via $-\widehat{\lambda}_j^*$, we use $\varphi(\widehat{\lambda}_j)$ for $j=1,\dots,r$. This modified version of IRKA allows to reduce transfer functions with poles in general domains that are characterized by a specific set of conformal maps (see Corollary \ref{corollary:practical} and \ref{corollary:practicaltheorem}). In Algorithm \ref{algoirka}, we provide an appropriate pseudocode.

\begin{algorithm}
	\caption{Iterative Rational Krylov Algorithm with Conformal Maps}\label{algoirka}
	\begin{algorithmic}[1]
		\Require The full order system matrices $(\mathbf{A},\mathbf{b},\mathbf{c})$, the conformal map $\psi$, and the reduced order $r<n$.
		\State Make initial guess of interpolation points $\boldsymbol{\sigma}_0$
		\State Construct the projection matrices $\mathbf{V}_r$ and $\mathbf{W}_r$ as in (\ref{eq:projectionmatrices})
		\While{$\|\boldsymbol{\sigma}_{i+1}-\boldsymbol{\sigma}_{i}\|/\|\boldsymbol{\sigma}_{i}\|>\mathsf{tol}$}
		\State Assign $\widehat{\mathbf{A}}_r = (\mathbf{W}_r^*\mathbf{V}_r)^{-1}\mathbf{W}_r^*\mathbf{A}\mathbf{V}_r$
		\State Solve the eigenvalue problem $\widehat{\mathbf{A}}_r \mathbf{v}_j = \widehat{\lambda}_j\mathbf{v}_j$ and set $\boldsymbol{\sigma}_{i+1}^{(j)}=\varphi(\widehat{\lambda}_j)$ \newline for $j=1,\dots,r$
		\State Update the projection matrices such that \newline $\text{Ran}(\mathbf{V}_r)=\text{span}\left\{(\boldsymbol{\sigma}_{i+1}^{(1)}\mathbf{I}-\mathbf{A})^{-1}\mathbf{b},\dots, (\boldsymbol{\sigma}_{i+1}^{(r)}\mathbf{I}-\mathbf{A})^{-1}\mathbf{b} \right\}$,\newline
		$\text{Ran}(\mathbf{W}_r)=\textnormal{span}\left\{((\boldsymbol{\sigma}_{i+1}^{(1)})^*\mathbf{I}-\mathbf{A}^*)^{-1}\mathbf{c},\dots, ((\boldsymbol{\sigma}_{i+1}^{(r)})^*\mathbf{I}-\mathbf{A}^*)^{-1}\mathbf{c} \right\}$
		\EndWhile
		\State Construct the reduced order model matrices $\widehat{\mathbf{A}}_r,\widehat{\mathbf{b}}_r, \widehat{\mathbf{c}}_r$ using (\ref{eq:systemmatrices}).
	\end{algorithmic}
\end{algorithm}

Let us emphasize that the optimal $\mathcal{H}_2(\bar{\mathbb{A}}^{\mathsf{c}})$ model reduction problem aims at minimizing the error $\|\mathfrak{H}_H-\mathfrak{H}_{\tilde{H}}\|_{\mathcal{H}_2(\mathbb{C}_+)}$ with respect to $\tilde{H}$. In particular, the conformal map $\psi$ might cause $\mathfrak{H}_H$ to have poles very close to the imaginary axis resulting in a potentially poor convergence behavior. An appropriate choice of $\psi$ might therefore require an individual analysis of the problem at hand. In the next section, we will report on such issues and also compare Algorithm \ref{algoirka} against the classical version of IRKA.

\section{Numerical experiments}\label{sec:numexp}

In this section, we test our theoretical results with two numerical examples that show the effectiveness of Algorithm \ref{algoirka} applied to systems that are not asymptotically stable.
The two cases considered are the Schr\"odinger and the undamped wave equation. Our main purpose is to show that Algorithm \ref{algoirka} is able to effectively reduce systems with poles along the imaginary axis, a case where the $\mathcal{H}_2(\mathbb{C_+})$ framework would fail. In the first example, we compare the results of Algorithm \ref{algoirka} against IRKA. In the second experiment, we apply a more complex conformal map, developed in Section \ref{sec:ellipse}, and show the performance of the resulting reduced model.

All simulations were generated with \textsc{MATLAB}\textsuperscript{\textregistered} 2023b on a laptop computer equipped with an Apple Silicon\textsuperscript{\textregistered} M2 Pro processor and 16GB of RAM. The $\mathcal{H}_2(\bar{\mathbb{A}}^\mathsf{c})$ error norms are computed through the \texttt{integral} command while the trajectories are a result of the \texttt{ode23} routine. For both the \texttt{integral} and \texttt{ode23} functions we used  relative and absolute tolerances of $10^{-8}$ and $10^{-12}$, respectively. The implementation is also publicly available\footnote{\url{https://github.com/aaborghi/H2-arbitrary-domains}}.

\subsection{Schr\"odinger equation}

In the first example, we consider the following boundary controlled Schr\"odinger equation (see, e.g., \cite[Example 6.7.3, Section 11.6.1]{TucW09})
\begin{equation*}
	\begin{aligned}
		\frac{\partial w(x,t)}{\partial t}&=-\mathrm{i} \frac{\partial^2w(x,t)}{\partial x^2}, && \text{on } (0,1)\times(0,T), \\
		w(0,t)&=0, \ \ w(1,t)=u(t), && \text{on } (0,T), \\
		y(t) &= \int_0^1 w(x,t)\,\mathrm{d}x, && \text{on } (0,T),\\
		w(x,0)&=0, && \text{in } (0,1),
	\end{aligned}
\end{equation*}
where $u$ and $y$ are the (scalar) input and output of the system.
We use a spatial semi discretization by centered finite differences resulting in a full order system of dimension $n=1000$. As this system has its poles on the upper part of the imaginary axis, we apply the following conformal map from Section \ref{sec:upperhalfplane} which rotates (clockwise) the left half plane by $\frac{\pi}{2}$
\begin{equation*}
	\psi(s) = -\mathrm{i}s.
\end{equation*}
This leads to the function in (\ref{eq:upperhalfplaneintpoints}) for computing the interpolation points. The initial shifts $\boldsymbol{\sigma}_0$ for Algorithm \ref{algoirka} are computed in \textsc{MATLAB} as follows
\[
\boldsymbol{\mu} = 500\;\texttt{randn}(r/2,1)-1000\mathrm{i}\;\texttt{rand}(r/2,1), \quad \boldsymbol{\sigma}_0 = \begin{bmatrix}
    \boldsymbol{\mu} \\ -\overline{\boldsymbol{\mu}}
\end{bmatrix}.
\]
For testing IRKA we chose the initial shifts equal to $\mathrm{i}\boldsymbol{\sigma}_0$. In Figure \ref{fig:schupperplane}, we depict the relative $\mathcal{H}_2(\bar{\mathbb{A}}^{\mathsf{c}})$ error defined as 
\begin{equation*}
	\frac{\|H-\widehat{H}\|_{\mathcal{H}_2(\bar{\mathbb{A}}^{\mathsf{c}})}}{\|H\|_{\mathcal{H}_2(\bar{\mathbb{A}}^{\mathsf{c}})}},
\end{equation*}
with $\bar{\mathbb{A}}^{\mathsf{c}}=\mathbb{C}_\downarrow$, see Section \ref{sec:H2Doptimalconditions}, and reduced orders varying from $r=4$ to $r=24$. Here, as expected, Algorithm \ref{algoirka} clearly outperforms IRKA with regard to the $\mathcal{H}_2(\bar{\mathbb{A}}^{\mathsf{c}})$ error. This is not surprising since, due to the poles of the system being located on the upper part of the imaginary axis, the system is not in $\mathcal{H}_2(\mathbb{C}_+)$. We also mention that, throughout the range of reduced orders $r$ in Figure \ref{fig:schupperplane}, Algorithm \ref{algoirka} converges with fewer iterations.  
Since Algorithm \ref{algoirka} first of all tries to ensure that the resulting reduced model has its poles also in the open upper half plane, the reduced poles generally cannot be expected to remain exactly on the imaginary axis, therefore resulting in possibly unstable systems. In Figure \ref{fig:schAlg1gaussinput} the trajectories of the full order model (FOM) and the resulting reduced order model (ROM) with $r=16$ are considered for a specific Gaussian input $u$ depicted in the bottom plot. We see that the output of the reduced model $\widehat{y}_r$ almost exactly replicates $y$ with a low relative error.
\begin{figure}
	\centering
	\includegraphics[width=0.55\textwidth]{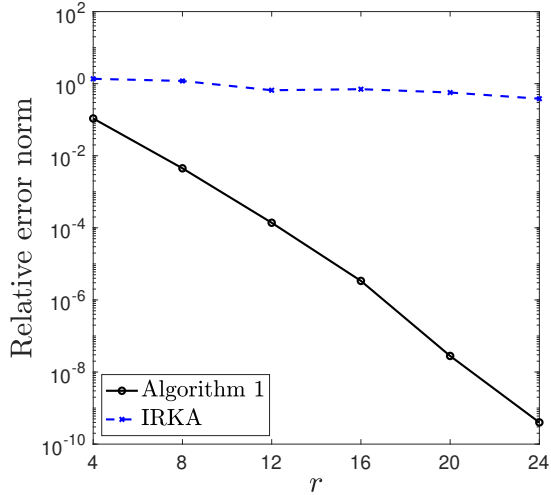}
	\caption{The $\mathcal{H}_2(\bar{\mathbb{A}}^{\mathsf{c}})$ relative error of Algorithm \ref{algoirka} and IRKA for different reduced orders $r$ and $n=1000$. }
	\label{fig:schupperplane}
\end{figure}
Figure \ref{fig:schAlg1gaussinput} also shows how the trajectories of a ROM computed by IRKA, with $r=16$, behaves under the same input as above. 
Looking at the relative errors, we see that IRKA outperforms Algorithm \ref{algoirka} for times where $u(t)\approx 0$. On the other hand, the model generated by IRKA appears to be less accurate when the control input becomes active.
It is important to note that, within the considered class of systems, the computation of the reduced model through IRKA is sensitive w.r.t. the selection of the initial shifts. 
Indeed, when varying the initial shifts we observed significant differences in the IRKA reduced models, ranging from highly unstable systems to very accurate stable systems that outperform Algorithm \ref{algoirka}. 
In any case, there is no theoretical foundation justifying that the model produced by IRKA is optimal in a specific sense as the full order systems are not in the $\mathcal{H}_2(\mathbb{C}_+)$ space. In addition, it is important to highlight that, due to the interpolation conditions, the shifts are set equal to the mirror image of the reduced model poles at each iteration. Since the Schr\"odinger equation has its spectrum on the imaginary axis, for larger reduced system dimensions, IRKA is likely to become ill-conditioned due to (numerical) singularities in the required rational Krylov subspaces.

\begin{figure}
    \centering
    \includegraphics[width=\textwidth]{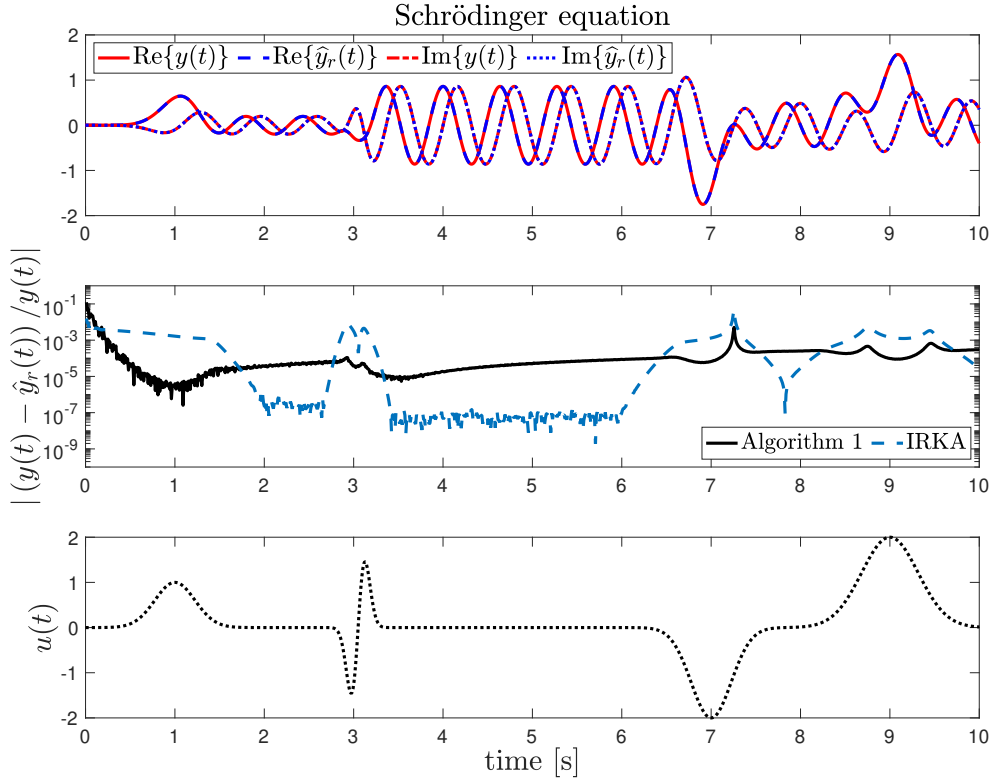}
    \caption{(Top) the real and imaginary output responses of the FOM ($y$) and ROM ($\widehat{y}_r$) for the chosen input $u$. Here the ROM system matrices are computed with Algorithm 1 for $r=16$. (Middle) the relative error of the reduced model computed with Algorithm 1 and IRKA. (Bottom) trajectory of the chosen control input $u$.}
    \label{fig:schAlg1gaussinput}
\end{figure}

\subsection{Wave equation} 

As a second example, we consider the linear undamped wave equation subject to distributed control and observation given by
\begin{equation}\label{eq:wave}
	\begin{aligned}
	\frac{\partial^2w(x,t)}{\partial t^2}&= \frac{\partial^2w(x,t)}{\partial x^2}+\chi_{[0.6,0.7]}u(t), && \text{on } (0,1)\times(0,T), \\
		w(0,t)&=0, \ \ w(1,t)=0, && \text{on } (0,T), \\
		y(t) &= \int_{0.1}^{0.4} w(x,t)\,\mathrm{d}x, && \text{on } (0,T),\\
		w(x,0)&=0, && \text{in } (0,1),
	\end{aligned}
\end{equation}
where $\chi_{[0.6,0.7]}$ denotes the indicator function on the interval $I=[0.6,0.7]$. Again, we employ a finite difference discretization with 5000 inner grid points leading to a first order ODE system of dimension $n=10000$. Here, the poles are located on the imaginary axis but they are now symmetrically distributed according to the real axis. For this example, we choose the conformal map described in Section \ref{sec:ellipse}. We utilize (\ref{eq:joukowskiextendedcM}) where we include a translation by $c\in\mathbb{C}$, and scaling and rotation by $M\in\mathbb{C}$. To include all the FOM poles we choose $c=-5\times10^{-3}$ and $M=1.5\times10^{4}\mathrm{i}$. 
To restrict the poles of the reduced model on the imaginary axis we choose $R=1+1\times10^{-6}$. This makes the minor axis $(R-R^{-1})/2$ approach 0 and so constraining Algorithm \ref{algoirka} to position the poles on the imaginary axis. However, because the poles of the FOM transfer function $H$ will be close to the boundary of $\mathbb{A}$, i.e., the ellipse, the poles of $\mathfrak{H}_H$ will get closer to the imaginary axis. This can lead to some numerical issues in the construction of the reduced model and the computation of the $\mathcal{H}_2(\bar{\mathbb{A}}^\mathsf{c})$ norm. 
The interpolation points are chosen according to (\ref{eq:ellipseintpointscM}) in each iteration of Algorithm \ref{algoirka}. Here, the initial shifts are taken with fixed real part at 0.1 and a scaled normally distributed random choice of the imaginary part. 
It is worth mentioning that the choice of parameters in (\ref{eq:joukowskiextendedcM}), and eventually in (\ref{eq:ellipseintpointscM}), strictly depends on the position of the system poles. This requires the user to have some knowledge regarding the location of the spectrum for the correct use of the conformal map. Figure \ref{fig:wave} shows the impulse response of the FOM and a reduced model of order $r=20$.
We see that the two trajectories almost match with low absolute error.

Even if in this example Algorithm \ref{algoirka} shows potentially good performance for low frequencies, it must be pointed out that this is dependent on the choice of the initial shifts. As mentioned above, the boundary of the ellipse is mapped into the imaginary axis by $\psi^{-1}$. Having the poles of the FOM very close to the boundary of the ellipse makes the computation of the reduced system more sensitive to the choice of the initial shifts. Nevertheless, with this approach, we can restrict the poles of the ROM to be almost on the imaginary axis.

\begin{figure}
	\centering
	\includegraphics[width=1\textwidth]{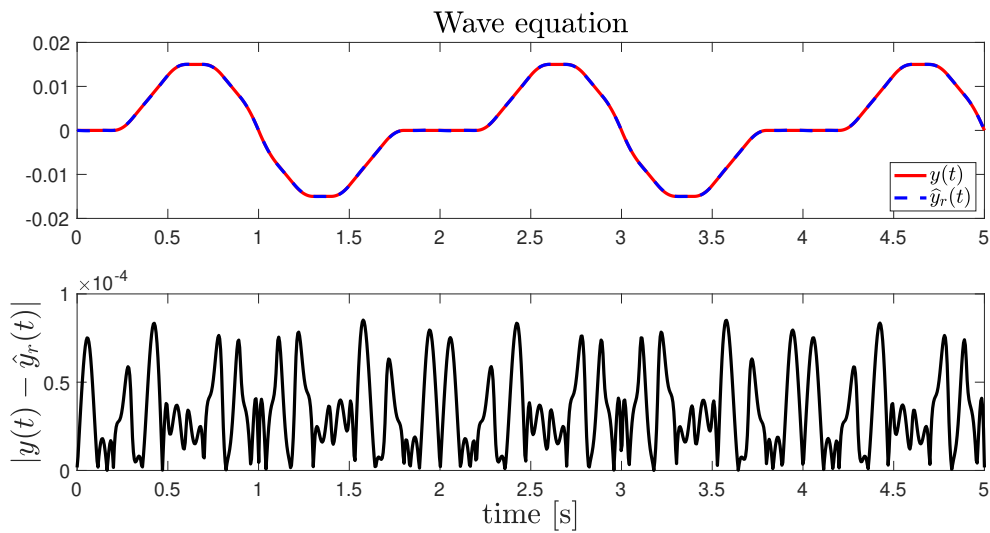}
	\caption{(Top) the output impulse response of the FOM ($y$) and ROM ($\widehat{y}_r$). Here the system dimension is $n=10000$ and the reduced order is $r=20$. (Bottom) the absolute error.}
	\label{fig:wave}
\end{figure}

\section{Conclusions}
In this paper, we introduced a novel $\mathcal{H}_2$ optimal model reduction framework that can treat transfer functions with poles in general domains. For this purpose, we used conformal maps to define the $\mathcal{H}_2(\bar{\mathbb{A}}^{\mathsf{c}})$ space and derived first order optimality conditions. 
With some additional assumptions, we retrieved simplified optimal interpolation conditions which we used to develop a modified version of IRKA. 

\paragraph{Acknowledgments}
We gratefully acknowledge the support of the Deutsche \\Forschungsgemeinschaft (DFG) as part of GRK2433 DAEDALUS (Project number 384950143). We would like to thank Olivier S\`ete, Jan Zur, and Mathias Oster for their insightful and helpful discussions.

\bibliographystyle{siamplain}
\bibliography{references}
\end{document}